\documentclass[10pt]{amsart}
\usepackage[T1]{fontenc}
\usepackage[latin1]{inputenc}
\usepackage{xypic,array}
\usepackage{amsfonts}
\usepackage{amsmath}
\usepackage{amssymb}
\usepackage{amsthm}
\usepackage{amssymb}
\usepackage{setspace}
\usepackage{tikz}
\usepackage{booktabs}
\usepackage{hyperref}
\usepackage{longtable}
\usepackage{geometry}
\usetikzlibrary{trees}
\usepackage{amsfonts}
\usepackage{amsmath}
\usepackage{epsfig}
\usepackage{color}
\usepackage{mathrsfs}
\usepackage{latexsym}
\usepackage{textcomp}
\usetikzlibrary{trees}
\usepackage{marginnote}
\usepackage{multirow}

\usepackage{fancyvrb} 

\newtheorem{Theorem}{Theorem}[section]
\newtheorem{mthm}{Main Theorem}
\newtheorem{Lemma}[Theorem]{Lemma}
\newtheorem{Prop}[Theorem]{Proposition}
\newtheorem{Def}[Theorem]{Definition}
\newtheorem{Cor}[Theorem]{Corollary}
\newtheorem{Ex}[Theorem]{Example}

\newtheorem{Obs}[Theorem]{Remark}

\newcommand{\quotf}{{\mathcal Q}{\it uot}}

\newcommand{\sch}{\mathfrak{Sch}_{/\C}}
\newcommand{\sets}{\mathfrak{Set}}

\def\C{{\mathbb{C}}}

\def\P{{\mathbb{P}}}

\def\Z{{\mathbb{Z}}}
\def\sF{{\mathscr{F}}}
\def\sD{{\mathscr{D}}}

\def\G{{\mathscr{G}}}

\def\O{{\mathcal{O}_{\mathbb{P}^3}}}
\def\OZ{{\mathcal{O}_{Z}}}
\def\IZ{{\mathcal{I}_{Z}}}
\def\OC{{\mathcal{O}_{C}}}
\def\IC{{\mathcal{I}_{C}}}

\def\OX{{\mathcal{O}_{X}}}

\def\G{{\mathscr{G}}}

\def\EE{{\mathcal{E}}}

\def\O{{\mathcal{O}_{\mathbb{P}^3}}}

\def\OX{{\mathcal{O}_{X}}}
\def\OZ{{\mathcal{O}_{Z}}}

\newcommand{\op}[1]{{\mathcal O}_{\mathbb{P}^{#1}}}
\newcommand{\p}[1]{{\mathbb{P}^{#1}}}
\newcommand{\tp}[1]{{\rm T}{\mathbb{P}^{#1}}}
\DeclareMathOperator{\rk}{{rk}}
\DeclareMathOperator{\coker}{{coker}}
\DeclareMathOperator{\ext}{{Ext}}
\DeclareMathOperator{\Hom}{{Hom}}
\newcommand{\inext}{{\mathcal E}{\it xt}}

\DeclareMathOperator{\sing}{Sing}

\DeclareMathOperator{\img}{{im}}

\DeclareMathOperator{\Pic}{{Pic}}

\title[ Classification of the invariants of foliations by curves of low degree on $\mathbb{P}^3$]
{Classification of the invariants of foliations by curves of low degree on the three-dimensional projective space}

\author[M. Corr\^ea]{M. Corr\^ea}
\thanks{ }
\dedicatory{}
\address{
Universit\`a degli Studi di Bari, 
Via E. Orabona 4, I-70125, Bari, Italy
}
\email{mauricio.barros@uniba.it}

\author[  M. Jardim]{M. Jardim}
\thanks{ }
\dedicatory{}
\address{IMECC - UNICAMP \\ Departamento de Matem\'atica \\
Rua S\'ergio  Buarque de Holanda, 651\\ 13083-970 Campinas-SP, Brazil}
\email{jardim@ime.unicamp.br}

\author[S.  Marchesi]{S.  Marchesi}
\thanks{}
\dedicatory{}
 \address{ Universitat de Barcelona \\
Facultat de Matem\`atiques i Inform\`atica
 \\
Gran Via de les 
Corts Catalanes 585
\\ 08007, Barcelona, 
Spain}
\address{Centre de Recerca Matem\`atica, Edifici C, Campus Bellaterra, 08193 Bellaterra, Spain}
\email{marchesi@ub.edu}
\dedicatory{}

\email{ }
\keywords{}
\subjclass{}
\date{}
\begin{document}

\begin{abstract}
We study foliations by curves on the three-dimensional projective space with no isolated singularities, which is equivalent to assuming that the conormal sheaf is locally free.
We provide a classification of the topological and algebraic invariants of the conormal sheaves and singular schemes for such foliations by curves, up to degree 3.  In particular, we prove that  foliations by curves of degree 1 or 2 are contained in a pencil of planes or are Legendrian, and are given by the complete intersection of two codimension one distributions. Furthermore, we prove that the conormal sheaf of a foliation by curves of degree 3 with reduced singular scheme either splits as a sum of line bundles or is an instanton bundle. For degree larger than 3, we focus on two classes of foliations by curves, namely Legendrian foliations and those whose conormal sheaf is a twisted null-correlation bundle. We give characterizations of such foliations, describe their singular schemes and their moduli spaces.
\end{abstract}

\maketitle
\tableofcontents
\bigskip

\section{Introduction}
The qualitative study of polynomial differential equations was initiated in classical works by  Poincar\'e, Darboux, and  Painlev\'e, see for instance \cite{Darboux,Pain,Poincare}. In the modern terminology, these works provided us with study of holomorphic  foliations by curves on $\mathbb{P}^2$ by analyzing their possible algebraic leaves.

Since then, classifications of codimension one foliations on higher dimensional projective spaces have also been obtained. To be more precise, Jouanolou classified codimension one foliations of degrees 0 and 1 in \cite{J}; Cerveau and Lins Neto showed in \cite{CL} that there exist six irreducible components of foliations of degree 2 on projective spaces and   in   \cite{CL3}  they  proved that foliations of degree three are either transversely affine foliations, or are rational pullbacks of foliations on $\mathbb{P}^2$. 

Recently, the authors of \cite{CCJ} initiated a systematic study of codimension one holomorphic distributions on $\mathbb{P}^3$, analyzing the properties of their singular schemes and tangent sheaves. In particular, a classification of codimension one distributions of degree at most 2 with locally free tangent sheaves was provided, together with a description of the geometry of certain the moduli space of distributions. 

By contrast, classifications for  distributions of codimension two are not widely known.  Since a distribution of codimention two is generated by a single twisted vector field, it is automatically integrable. In addition, its leaves are 1-dimensional complex submanifolds, hence it is called a \textit{foliations by curves}. This is a misnomer, though, since the leaves are not, in general, algebraic curves. 

Note that generic foliations by curves, understood as a twisted holomorphic vector fields, have only isolated singularities, and the challenge is to understand foliations by curves with non-isolated singularities. In this work we are interested on the class of non-generic foliations whose singular schemes are of pure dimension one. We say that such foliations are of {\it local complete intersection type}, since they are given by twisted 2-forms which are locally decomposable even along their singular scheme, i.e, the foliation is given locally by the intersection of two codimension one distributions. This also means that the conormal sheaf is locally free, see Lemma \ref{n-loc free}. Therefore, in order to provide a clasification of the topological and algebraic invariants of locally complete intersection foliations it is sufficient to describe the geometry of their conormal bundles and their singular schemes in the spirit of \cite{CCJ}.  Analogoulsy, we say that a foliation is of {\it global complete intersection type} if it is defined globally as the intersection of two codimension one distributions.

 We present a general theory of foliations by curves on non-singular projective varieties in Section \ref{sec:basics}. Next, we construct the moduli space of foliations by curves in Section \ref{sec:moduli}, proving certain criteria that will later allow us to check whether certain specific moduli spaces of foliations by curves are irreducible and smooth; in particular, see Theorem \ref{quot-thm} and Lemma \ref{lemma-moduli}.

 We then focus on foliations by curves on $\p3$; these are given by short exact sequences of the form
$$ \sF ~:~ 0 \longrightarrow N_\sF^* \longrightarrow \Omega^1_{\p3} \stackrel{\phi^\vee}{\longrightarrow} \IZ(d-1) \longrightarrow 0, $$
where $N_\sF^*$ is the conormal sheaf and $Z$ is the singular scheme of $\sF$; the integer $d\ge0$ is the degree of $\sF$. As mentioned above, $\sF$ is of global complete intersection type if and only if $N_\sF^*$ is locally free, if and only if $Z$ has pure dimension 1. Under this hypothesis, one can show (see Lemma \ref{bound c2}) that
$$ d+2\le c_2(N_\sF^*) \le d^2+2d+1. $$
The \textit{topological classification} of foliations by curves of a given degree $d$ is to determine the integers $c$ in the interval above for which there exists a foliation $\sF$ of degree $d$ such that $c_2(N_\sF^*)=c$, and then characterize the conormal sheaf $N_\sF^*$ and the singular scheme $Z$. A \textit{full classification} of foliations by curves of a given degree $d$ is achieved by describing the irreducible components of its moduli space.

 Our first goal is to provide the topological classification of locally complete intersection foliations by curves of degree at most 3 on $\P^3$.
Since there are no locally complete intersection foliations by curves of degree 0 (see the first paragraph of Section \ref{sec:deg12}), our study starts in degrees 1 and 2. In these cases, it turns out that locally complete intersection foliations are actually globally complete intersection, and these can be fully classified using \cite[Theorem 5.1]{CJV} (reformulated as Theorem \ref{determi} below). Therefore, a full classification of locally complete intersection foliations by curves of degrees 1 and 2 is achieved by showing that the moduli spaces of such foliations are irreducible quasi-projective varieties, and their dimensions are computed.


We remark that in all statements below by a \emph{curve} we mean a closed subscheme of $\p3$ of pure dimension 1  (keep in mind, however, that the leaves of a foliation by curves are not curves in this sense, as observed above).

\begin{mthm}\label{classification-degree12}
Let $\sF$ be a foliation by curves on $\P^3$ of degree $d\in\{1,2\}$. If $\sing(\sF)$ is a  curve, then its conormal sheaf $N_\sF^*$  splits as a sum of line bundles. More precisely, we have that 
\begin{enumerate}
\item if $d=1$, then $N_\sF^*\simeq\O(-2)^{\oplus 2}$ and $\sing(\sF)$ consists of two disjoint lines;
\item if $d=2$, then $N_\sF^*\simeq\O(-2)\oplus\O(-3)$
and $\sing(\sF)$ is a connected curve of degree $5$ and arithmetic
genus $1$;
\item 
  if   $\mathscr{F}'$ is a foliation by curves  on  $\mathbb{P}^3$, with   degree  $d\in\{1,2\}$, such that $\sing(\mathscr{F}) \subset \sing(\mathscr{F}') $, then $ \mathscr{F}'=\mathscr{F}$.
\end{enumerate}
In particular, $\sF$ is contained in a pencil of planes or is Legendrian, and it is given by the complete intersection of two codimension one distributions.

In addition, the moduli space of locally complete intersection foliations of degree $d\in\{1,2\}$ is an irreducible quasi-projective variety of dimension 8 if $d=1$, and dimension $20$  if $d=2$.
\end{mthm}

 Our next result provides the topological classification of locally complete intersection foliations by curves of degree 3; we find examples of foliations with conormal sheaves that do not split as a sum of line bundles. Recall that an \emph{instanton bundle} on $\p3$ is a stable rank 2 locally free sheaf $E$ satisfying $h^1(E(-2))=0$; $c_2(E)$ is called the \emph{charge} of $E$. Moreover, $E$ is said to be a \emph{'t Hooft} instanton bundle if, $h^0(E(1))\ge1$, and a \emph{special 't Hooft} instanton bundle if, in addition, $h^0(E(1))=2$, see \cite{HN}. 

\begin{mthm}\label{classification-degree3} 
Let $\sF$ be a foliation by curves on $\P^3$ of degree $3$. If $\sF$ is of local complete intersection type, then one of the following possibilities hold:
\begin{enumerate}
\item $N_\sF^*=\O(-2)\oplus\O(-4)$, and $\sing(\sF)$ is a connected curve of degree $10$ and arithmetic genus $5$;
\item $N_\sF^*=\O(-3)^{\oplus 2}$, and $\sing(\sF)$ is a connected curve of degree $9$ and arithmetic genus $3$; 
\item $N_\sF^*=E(-3)$, where $E$ is stable rank 2 locally free sheaf with $c_1(E)=0$ and $1 \leq c_2(E) \leq 5$; the singular scheme $\sing(\sF)$ is a curve of degree $9-c_2(E)$ and arithmetic genus $p_a(C)=8-3c_2(E)$.
\end{enumerate}
If, in addition, $\sing(\sF)$ is reduced, then $1 \leq c_2(E) \leq 4$, $E$ is an instanton bundle (though not a special 't Hooft instanton bundle of charge 3 or a 't Hooft instanton bundle of charge 4), and $\sing(\sF)$ is connected if and only if $c_2(E)=1,2$.
 
Conversely, for each $n\in\{1,2,3,4,5\}$, there is foliation by curves $\sF$ of degree 3 on $\p3$ such that $N_\sF^*(3)$ is an instanton bundle of charge $n$. 
\end{mthm}

We observe that \cite[Theorem 5.1]{CJV} states that a globally complete intersection foliation by curves in $\P^3$ of degree $d\geq 1$ is completely determined by its conormal sheaf and singular scheme, see also Theorem \ref{determi} below. In light of the stated results, we have that a locally complete intersection foliation of degree 1 or 2 is determined by its singular scheme. On the other hand, there are locally complete intersection foliations of degree 3 which are not globally complete intersections, and these are not determined by their singular schemes, as it is illustrated in Example \ref{example5lines}.  

One consequence of our  topological classification is that degree and genus of the singular scheme are not enough to distinguish the possible foliations of local complete intersection type which are not of global complete intersections. The new invariant that comes into play is the cohomology module of the conormal sheaf
$$ M_{\sF}:= H^1_*(N_\sF^*):=\bigoplus_{p\in\Z} H^1(N_\sF^*(p)), $$
regarded as a graded $\C[x_0,x_1,x_2,x_3]$-module,  where $x_0,x_1,x_2,x_3$ are homogeneous
coordinates on $\mathbb{P}^3$. 

\begin{table}[h]
	\centering
	\begin{tabular}{ |c | c | c | c | c | c | }
		\hline
		 $(\deg(C),p_a(C)) $  &  $ c_2(N_\sF^*(3))$  & $\dim M_{\sF}$ & $h^0(\OC)$    \\
		\hline \hline
\cline{2-4}
		(8,5) & 1   & 1 &   1  \\
			\cline{2-4}
		\hline \hline
		 $(7,2)$ & $2$  &  4 & 1 \\
		\hline \hline
		\multirow{2}{0.75cm}{\centering $(6,-1)$}    &   \multirow{2}{0.75cm}{\centering $3$}     & 8 &    2 \\
		\cline{3-4}
		&    &  $ 9 $ & 3  \\ \hline \hline
	$(5,-4)$	&  $4$     & 14  &  5  \\
		\hline
	\end{tabular}
	\medskip
	\caption{ Classification of the topological and algebraic invariants of foliations of degree 3 which are not global complete intersection, with reduced singular scheme. }
	\label{table deg 3}
\end{table}
 
Regarding the existence part of Main Theorem \ref{classification-degree3}, we prove a somewhat stronger statement: every null-correlation bundle (instanton of charge  $1$) arises as the conormal sheaf, up to twist, of a foliation by curves of degree 3; see Propositions \ref{c_2=1}  below. 

Beyond degree 3, we focus on two particular classes of foliations by curves. First, we consider the so-called \emph{Legendrian foliations}; a foliation by curves is called \emph{Legendrian} if it is a sub-distribution of a contact distribution on $\P^3$, see details in Section \ref{sec:legendrian}. We prove that such foliations are globally complete intersections, and establish the following characterization.

\begin{mthm}\label{mthm3}
Every Legendrian foliation $\sF$ of degree $d$ is of the form $\omega_0\wedge\omega$, where $\omega_0$ is a contact form and $\omega\in H^0(\Omega^1_{\p3}(d+1))$. In addition, the moduli space of the Legendrian foliations of degree $d$ is an irreducible quasi-projective variety of dimension
$$d\cdot\binom{d+3}{2} - \binom{d+2}{3}+4\:\:\mbox{ if } \:\: d \geq 2$$
and of dimension 8 if $d=1$.
\end{mthm}

We remark that there exists only one contact form $\omega_0\in H^0(\Omega^1_{\p3}(2))$,  up to automorphism of $\p3$.

Finally, we consider those locally complete intersection foliations by curves whose conormal sheaf is a twisted null-correlation bundle, which is the simplest rank 2 locally free sheaf on $\P^3$ which does not split as a sum of line bundles.

These foliations can also be regarded as the simplest locally complete intersections foliations which are not globally complete intersections from an algebraic point of view. Indeed, one important algebraic invariant of a foliation is the \emph{Rao module} of its singular scheme $Z$, namely the graded $\C[x_0,x_1,x_2,x_3]$-module
which is closely related with the graded module $H^1_*(N_\sF^*)$ mentioned above. One can prove that the conormal sheaf of a foliation by curves splits as a sum of line bundles if and only if the Rao module of singular scheme is 1-dimensional over $\C$, see \cite[Theorem 2]{CJV}. Our last main result shows that the foliations by curves whose singular scheme has a 2-dimensional Rao module are precisely the ones for which the conormal sheaf is a null-correlation bundle, up to twist.

\begin{mthm}\label{mthm4}
Let $\sF$ be a foliation by curves on $\mathbb{P}^3$. The conormal sheaf is a twisted null-correlation bundle if and only if the singular scheme is a curve with 2-dimensional Rao module. Furthermore, the moduli space of foliations by curves of degree $2k+1$ $(k\ge1)$ whose conormal sheaf is a twisted null-correlation bundle is an irreducible quasi-projective variety of dimension 
$$ 8 \binom{k+4}{3} - 2\binom{k+5}{3} -3k - 3. $$
The singular scheme of such foliations is always connected, and also smooth for a generic one.
\end{mthm}

\subsection*{Acknowledgments}
We would like to thank Israel Vainsencher and Alan Muniz for useful discussions. 
MC was supported by the CNPQ grants number 202374/2018-1, 302075/2015-1, and 400821/2016-8; he was also supported by the FAPESP grant number 2015/20841-5; he is supported by the 
Prin 2022 Interactions between Geometric Structures and Function Theories;  he is grateful to Universidade Estadual de Campinas for its hospitality; he would like to thank the Universidade  Federal de Minas Gerais, the institution he was affiliated with when this work started. 
MJ is supported by the CNPQ grant number 302889/2018-3 and the FAPESP Thematic Project 2018/21391-1. 
SM was partially supported by FAPESP grant number 2019/08279-0, by CNPQ grant number 303075/2017-1, by PID2020-113674GB-I00, and by the Spanish State Research Agency, through the Severo Ochoa and Mar\'ia de Maeztu Program for Centers and Units of Excellence in R$\&$D (CEX2020-001084-M) and is a member of INdAM-GNSAGA. SM would like to thank the Universidade Estadual de Campinas, the institution he was affiliated with when this work started. This work was also partially funded by CAPES - Finance Code 001. 


\section{Foliations by curves}\label{sec:basics}

Let $X$ be a nonsingular projective variety of dimension $n$. Recall that a \emph{codimension $r$ distribution} $\sF$ on $X$ is given by an exact sequence
\begin{equation}\label{eq:Dist}
\mathscr{F}:\  0  \longrightarrow T_\sF \stackrel{\phi}{ \longrightarrow} TX \stackrel{\pi}{ \longrightarrow} N_{\sF}  \longrightarrow 0,
\end{equation}
where $T_\sF$ is a coherent sheaf of rank $s:=n-r$, and $N_{\sF}$ is a torsion free sheaf. The sheaves $T_\sF$ and $N_{\sF}$ are called the \emph{tangent} and the \emph{normal} sheaves of $\mathscr{F}$, respectively. Note that $T_\sF$ must be reflexive \cite[Proposition 1.1]{H2}.

The \emph{singular scheme} of $\mathscr{F}$ is defined as follows. Taking the maximal exterior power of the dual morphism $\phi^\vee:\Omega^1_X\to T_\sF^*$ we obtain a morphism $\Omega^{s}_X\to \det(T_\sF)^*$;  the image of the induced   morphism $\Omega^{s}_X\otimes \det(T_\sF)\to \mathcal{O}_X$ is an ideal sheaf $\IZ$ of a subscheme $Z\subset X$, which is called the singular scheme of $\mathscr{F}$. 

Finally, we introduce the notion of integrability. A \emph{foliation} is an integrable distribution, which means a distribution whose tangent sheaf is closed under the Lie bracket of vector fields, that is $[\phi(T_\sF),\phi(T_\sF)]\subset \phi(T_\sF)$.

In this paper we focus on the case $r=n-1$. Clearly, every distribution of codimension $n-1$ is integrable, and it is called a \emph{foliation by curves}. In addition, $T_\sF$ must be a line bundle on $X$, which we denote by $\mathscr{L}$ from now on, while the normal sheaf $N_\sF$ is a torsion-free sheaf of rank $n-1$. Therefore, a foliation by curves is simply given by a nontrivial section $\phi\in H^0(TX\otimes\mathscr{L}^*)$, whose cokernel is a torsion free sheaf. 

Dualizing the sequence \eqref{eq:Dist}, we obtain
\begin{equation}\label{dual dist}
0 \to N_\sF^* \to \Omega_X^1 \stackrel{\phi^\vee}{\rightarrow} \mathscr{L}^* \to \inext^1(N_\sF,\OX) \to 0,
\end{equation}
thus $\inext^1(N_\sF,\OX)\simeq \OZ\otimes\mathscr{L}^*$, where $Z$ is the singular scheme of $\sF$. In other words, the singular set of $N_\sF$ coincides with the singular locus of $\sF$ as a set, which can also be regarded as the vanishing locus of $\phi$ as a section in  $H^0(TX\otimes\mathscr{L}^*)$. We also conclude that $\inext^p(N_\sF,\OX)=0$ for $p\ge 2$.

Cutting sequence \eqref{dual dist}, we obtain the following short exact sequence
\begin{equation}\label{pfaff}
0 \to N_\sF^* \to \Omega_X^1 \stackrel{\phi^\vee}{\rightarrow} \IZ\otimes\mathscr{L}^* \to 0,
\end{equation}
which will play an important role in this paper; the sheaf $N_\sF^*$ is called the \emph{conormal sheaf} of the foliation $\sF$.  

Conversely, we dualize the sequence in display \eqref{pfaff} to obtain
\begin{equation}\label{dual pfaff}
0 \to \mathscr{L} \stackrel{\phi}{\rightarrow} TX \to N_\sF^{**} \to \inext^1(\IZ,\OX)\otimes\mathscr{L}
\to 0.
\end{equation}
Since $N_\sF=\coker\phi$ by definition, we recover the original foliation by curves
$$ 0 \to \mathscr{L} \stackrel{\phi}{\rightarrow} TX \to N_\sF \to 0, $$
and conclude that
\begin{equation}\label{n-ndd}
0 \to N_\sF \to N_\sF^{**} \to \inext^2(\OZ,\OX)\otimes\mathscr{L} \to 0.
\end{equation}

Note that $Z$ might not be pure dimensional; let $R$ be the maximal subsheaf of $\OZ$ of codimension greater than $2$; the quotient $\OZ/R$ is the structure sheaf of a (possibly empty) scheme of pure codimension 2, which we denote by $C$. These facts are described in the short exact sequence
\begin{equation} \label{sqc:OC}
0 \to R \to \OZ \to \OC \to 0 . 
\end{equation}
It follows that $\inext^2(\OZ,\OX)\simeq \inext^2(\OC,\OX)=\omega_C\otimes\omega_X$, where $\omega_C$ and $\omega_X$ are the dualizing sheaves, and $\inext^p(\OZ,\OX)\simeq \inext^p(R,\OX)$ for $p\ge3$. This observation has {\color{cyan}two} interesting consequences. First, the sequence in display \eqref{n-ndd} can be rewritten in the following manner
\begin{equation}\label{n-ndd2}
0 \to N_\sF \to N_\sF^{**} \to \omega_C\otimes\omega_X\otimes\mathscr{L} \to 0.
\end{equation}

 Furthermore, we obtain the characterization stated in the following lemma.

\begin{Lemma}\label{n-loc free}
Let $\sF$ be a foliation by curves on a smooth projective variety $X$. $N_\sF^*$ is locally free if and only if its singular locus has pure codimension 2.
\end{Lemma}

\begin{proof}
The dualization of the sequence in display \eqref{pfaff} also leads to the isomorphisms
$$ \inext^p(N_\sF^{*},\OX) \simeq \inext^{p+1}(\IZ,\OX)\otimes\mathscr{L}\simeq \inext^{p+2}(R,\OX)\otimes\mathscr{L}. $$
The rightmost sheaf vanishes for all $p\ge1$ if and only $R=0$, which is equivalent to saying that $Z$ has pure codimension 2.
\end{proof}


\section{Moduli spaces of foliations by curves}\label{sec:moduli}

In \cite[Section 2.3]{CCJ} the authors developed a general construction of the moduli spaces of distributions on projetive varieties. The present section is dedicated to a slightly different \emph{dual} construction more suitable to understand foliations by curves. 

The set of all foliations by curves with a fixed line bundle $\mathscr{L}$ as tangent sheaf is simply the open subset  of \emph{saturated} sections $\phi\in H^0(TX\otimes\mathscr{L}^*)$, that is
$$ H^0(TX\otimes\mathscr{L}^*)^{\rm sat} := \{ \phi\in H^0(TX\otimes\mathscr{L}^*) ~|~ \coker\phi \textrm{ is torsion free } \}. $$
This set can be stratified according with the Hilbert polynomial of the vanishing locus of $\phi$. With this in mind, let $P$ be a fixed polynomial of degree at most $\dim X-2$; we define the set
$$ \mathcal{D}^P_{\mathscr{L}} :=  \{ \phi\in H^0(TX\otimes\mathscr{L}^*)^{\rm sat} ~|~ P_{\mathcal{O}_{Z_\phi}\otimes\mathscr{L}^\vee}(t)=P \}, $$
where $Z_\phi$ is the vanishing locus of $\phi$, that is  $\mathcal{O}_{Z_\phi}=\coker \phi^\vee $. Note that $\mathcal{D}^P_{\mathscr{L}}$ can be regarded as a locally closed subscheme of $H^0(TX\otimes\mathscr{L}^*)$. Here, $P_F(t)$ denotes the Hilbert polynomial of the sheaf $F$ on $X$. 

The set $\mathcal{D}^P_\mathscr{L}$ can be given an alternative description in terms of the Grothendieck quot-scheme for the cotangent bundle $\Omega^1_X$. Let us briefly recall its definition, using \cite[Section 2.2]{HL} as main reference.

Let $\sch$ denote the category of schemes of finite type over $\C$, and $\sets$ be the category of sets. Fix a polynomial $P\in\mathbb{Q}[t]$, and consider the functor
$$ \quotf^P : \sch^{\rm op} \to \sets ~~,~~ \quotf^P(S) := \{(N,\eta)\}/\sim $$
where
\begin{itemize}
\item[(i)] $N$ is a coherent sheaf of $\mathcal{O}_{X\times S}$-modules, flat over $S$, such that the Hilbert polynomial of $N_s:=N|_{X\times\{s\}}$ is equal to $P$ for every $s\in S$;
\item[(ii)] $\eta:\pi^*_X \Omega^1_X \to N$ is an epimorphism, where $\pi_X:X\times S\to X$ is the standard projection onto the first factor.
\end{itemize}
In addition, we say that $(N,\eta)\sim(N',\eta')$ if there exists an isomorphism $\gamma:N\to N'$ such that $\gamma\circ\eta=\eta'$.

Finally, if $f:R\to S$ is a morphism in $\sch$, we define $\quotf^P(f):\quotf^P(S)\to\quotf^P(R)$ by
$(N,\eta) \mapsto (f^*N,f^*\eta)$. Elements of the set $\quotf^P(S)$ will be denoted by $[N,\eta]$.

Let us recall the following result, which is just an adaptation of \cite[Theorem 2.2.4 and Proposition 2.2.8]{HL} suitable for our purposes.

\begin{Theorem}\label{quot-thm}
The functor $\quotf^P$ is represented by a projective scheme $\mathcal{Q}^P$ of finite type over $\C$, that is, there exists an isomorphism of functors $\quotf^P \stackrel{\sim}{\longrightarrow} \Hom(\cdot,\mathcal{Q}^P)$. 
In addition, if \linebreak $\ext^1(\ker\eta,N)=0$, then $\mathcal{Q}^P$ is smooth at a point $[N,\eta]$, and $\dim T_{[N,\eta]}\mathcal{Q}^P = \dim\Hom(\ker\eta,N)$.
\end{Theorem}

We assume from now on that $\Pic(X)=\mathbb{Z}\cdot\OX(1)$, so that the isomorphism class of a line bundle on $X$ is uniquely determined  by an integer called its degree in the sequel; for each $d\in\Z$, let $P_d:=P_{\OX(d)}$ and set $\mathcal{D}^P_{d}:=\mathcal{D}^P_{\OX(d)}$. 

We argue that there is a set theoretical bijection between $\mathcal{D}^P_{-d}$ and the open subset of $\mathcal{Q}^{P'}$, where $P':=P_{d}-P$, consisting of those pairs $[L,\eta]\in\mathcal{Q}^{P'}$ such that $L$ is a (rank 1) torsion free sheaf. Indeed, take $\phi\in\mathcal{D}^P_{-d}$; the sheaf $\img(\phi^\vee)=\IZ\otimes\OX(d)$ (compare with the exact sequence in display \eqref{pfaff}) is a quotient of $\Omega^1_X$ whose Hilbert polynomial is precisely $P'$ as above. Conversely, given a quotient $\eta:\Omega^1_X\to L$ with $L$ being a torsion free sheaf with $P_L(t)=P'$; it follows that $\rk(L)=1$, thus $L^*$ is a line bundle and $\eta^\vee\in H^0(TX\otimes L^{**})$.

From now on, we will regard the set $\mathcal{D}^P_{-d}$ as a scheme with the schematic structure inherited from the quot-scheme $\mathcal{Q}^{P'}$. 

Let $\mathcal{D}^{P,st}_{-d}$ denote the open subset of $\mathcal{D}^{P}_{-d}$ consisting of foliations by curves on $X$ whose conormal sheaf is $\mu$-stable. Let also $M^{P'}$ denote the moduli space of reflexive sheaves on $X$ with Hilbert polynomial equal to $P'$.  Following the same ideas and proofs as in \cite[Section 2.3]{CCJ}, especially Lemmas 2.5 and 2.6 there, we obtain the following statement.

\begin{Lemma}\label{lemma-moduli}
Let $X$ be a non singular projective variety of dimension $n$ with rank 1 Picard group, and let $P$ be polynomial of degree at most $n-2$. There exists a forgetful morphism
$$ \varpi:\mathcal{D}^{P,st}_{-d}\to M^{P_d-P} ~~,~~ [L,\eta] \mapsto \ker\eta, $$
sending a foliation by curves to its conormal sheaf. In addition, if $N_\sF^*=\ker\eta$ satisfies \linebreak $\ext^1(N_\sF^*,\Omega^1_X)=\ext^2(N_\sF^*,N_\sF^*)=0$, then $[L,\eta]$ is a non-singular point of $\mathcal{D}^{P,st}$, $\varpi$ is a submersion, and 
$$ \dim_{[L,\eta]}\mathcal{D}^{P,st}_{-d} = \dim \ext^1(N_\sF^*,N_\sF^*) + \dim \Hom(N_\sF^*,\Omega^1_X) -1 . $$
\end{Lemma}

\bigskip

Note that $\dim\ext^1(N_\sF^*,N_\sF^*)$ is precisely the dimension of $M^{P'}$ at the isomorphism class of $N_\sF^*$, while $\dim \Hom(N_\sF^*,\Omega^1_X) -1$ gives the dimension of the set of monomorphisms $N_\sF^*\to \Omega^1_X$ with torsion free cokernel, up to scalar multiplication. Therefore, a family of foliations by curves of the form
$$ \sF ~:~ 0 \to N_\sF^* \to \Omega^1_X \to \IZ\otimes\OX(d) \to 0 ~~{\rm where} ~~ Z:=\sing(\sF) $$
satisfying the two vanishing conditions on the previous lemma forms an irreducible component of $\mathcal{D}^{P,st}_{-d}$, understood as the moduli space of foliations by curves with stable conormal sheaf.


\section{Foliations by curves on $\p3$}\label{basics-p3}

From now on we will only consider foliations by curves in $X=\p3$. The sequence in display \eqref{pfaff} then simplifies to
\begin{equation}\label{pfaff2}
\sF ~:~ 0 \longrightarrow N_\sF^* \longrightarrow \Omega^1_{\p3} \stackrel{\phi^\vee}{\longrightarrow} \IZ(d-1) \longrightarrow 0,
\end{equation}
with $\phi\in H^0(T\p3(d-1))$, where $d\ge0$ is called the degree of $\sF$. The sheaf $R$ 
defined by the sequence in display \eqref{sqc:OC} is a sheaf of dimension 0, while the scheme $C$ is a curve; we will often denote it by $\sing_1(\sF)$, the 1-dimensional component of the singular locus of the foliation $\sF$.

From the Euler sequence, $\phi$  is represented by a  homogeneous polynomial vector fields $\widetilde{\phi} = \sum_{i=0}^3 F_i\frac{\partial}{\partial x_i}$ with $\deg(F_i) = d$  on $\mathbb{C}^4$ and the singular scheme of $\sF$  is given by the homogeneous ideal generated by the $2\times 2$ minors of the matrix	

\begin{equation}\label{eq:minors}
		\begin{bmatrix}
			F_0 & F_1 & F_2 & F_3 \\ x_0 & x_1 & x_2 & x_3
		\end{bmatrix}.
	\end{equation}


Our first step towards a deeper understanding of foliations by curves in $\p3$ is to determine a relation between the Chern classes of the conormal sheaf and the numerical invariants of the singular scheme.

\begin{Theorem}\label{P:SLocus}
Let $\sF$ be a foliation by curves of degree $d$ with $C$ and $R$ as defined by the sequence in display \eqref{sqc:OC}. One has
\begin{itemize}
\item[(i)] $c_1(N_\sF^*)=-3-d$;
\item[(ii)] $c_2(N_\sF^*) = d^2+2d+3-\deg(C)$;
\item[(iii)] $c_3(N_\sF^*) = h^0(R) = d^3 + d^2 + d + 1 - 3\deg(C)(d-1) - 2\chi(\OC)$.
\end{itemize}
\end{Theorem}

We observe that in \cite{CFNV,Gil,Vai} the authors determine the number of isolated singularities under the hypothesis that $R$ is the structure sheaf of a 0-dimensional scheme disjoint from $C$.

\begin{proof}
Consider the exact sequence (\ref{pfaff2}), and use $c( \Omega_{\P^3}^1)=c(N_\sF^*)\cdot c(\IZ(d-1))$ to obtain
\begin{equation}\label{eq:IT}
\begin{array}{rll}
-4 & = & c_1(N_\sF^*)+c_1(\IZ(d-1))\\
6 & = & c_1(N_\sF^*)\cdot c_1(\IZ(d-1))+c_2(N_\sF^*)+c_2(\IZ(d-1))\\
-4 & = & c_3(N_\sF^*)+c_3(\IZ(d-1))+c_1(N_\sF^*)\cdot c_2(\IZ(d-1))+c_2(N_\sF^*)\cdot c_1(\IZ(d-1))
\end{array}
\end{equation}

The first equality gives $c_1(N_\sF^*)=-3-d$, since $c_1(\IZ(d-1))=d-1$.

Since $c_2(\IZ(d-1))=\deg(C)$, the substitution of the values of the first Chern classes in the second equation, implies
$$ c_2(N_\sF^*)= d^2+2d+3-\deg(C). $$

Moreover, the substitution from the values of the first and second Chern classes in the third equation, we obtain
\begin{equation}\label{star}
d^3+d^2+d+1+c_3(\IZ(d-1)) + c_3(N_\sF^*) - 3d\deg(C) = 0.
\end{equation}
On the other hand, we have that
\begin{equation}\label{ast1}
c_3(\IZ(d-1))=c_3(\IZ)-(d-1)\deg(C) , ~~ {\rm and}
\end{equation}
\begin{equation}\label{ast2}
c_3(\IZ) = c_3(\IC) - c_3(\mathcal{O}_R) = 4\deg(C) - 2 \chi(\OC) - 2h^0(R).
\end{equation}
The substitution of the expression in display \eqref{ast1} and \eqref{ast2} in equation \eqref{star}, together with the fact that $c_3(N_\sF^*)=h^0(R)$ leads to
$$ h^0(R)=d^3+d^2+d+1 -3\deg(C)(d-1)-2\chi(\OC), $$
as claimed.
\end{proof}

Let us analyse two extreme situations. First, if the foliation $\sF$ has only isolated singularities, that is $\OZ=R$, then $N_{\sF}$ is reflexive by Lemma \ref{n-loc free} with Chern classes 
\begin{equation}\label{eq:ising}
c_2(N_{\sF}) = d^2+2d+3 ~~~{\rm and}~~~
c_3(N_{\sF}) = d^3+d^2+d+1. 
\end{equation}

On the other hand, if $Z$ has pure dimension 1, that is $R=0$, then $N_\sF^*$ is locally free by Lemma \ref{n-loc free} and one obtains the following expressions for the degree and arithmetic genus of $C$ in terms of the degree of the distribution and the second Chern class of the conormal sheaf:
\begin{eqnarray}\label{eq-deggenus}
\label{eq:degZ} \deg(C) & = & d^2 + 2d + 3 - c_2(N_\sF^*) \\
\label{eq:gen} p_a(C) & = & d^3 + d^2 + d -3(d-1)c_2(N_\sF^*)/2 - 4.
\end{eqnarray}

\begin{Lemma}\label{bound c2}
If $\sF$ is a foliation by curves of degree $d$ on $\p3$, then
$$ d+2 \le c_2(N_\sF^*) \le d^2 + 2d + 3. $$
If, in addition, $N_\sF^*$ is locally free, then $c_2(N_\sF^*) \le d^2 + 2d + 1$.
\end{Lemma}

\begin{proof}
The upper bound follows from the second equality in Theorem \ref{P:SLocus} by noticing that $\deg(C)\ge0$. The equality is attained when $\sF$ is a generic foliation, so that $\deg(C)=0$.
Assume that $\sF$ is not generic, i.e, $C\neq \emptyset$.
 It follows from \cite[Theorem 1.1]{Sancho} that $$\deg(C) \le d^2+d+1;$$ substituting this in the second equality in Theorem \ref{P:SLocus} gives the lower bound in the statement of the lemma.

If $N_\sF^*$ is locally free, then $\deg(C)\ge1$; if equality holds, it follows that $C$ must be a line, so $p_a(C)=0$. The equality in display \eqref{eq:gen} would imply that the polynomial equation
$$ d^3 + d^2 - 2d + 2 = 0 $$
has an integer solution, which it does not. Thus $\deg(C)\ge2$, and we obtain the improved upper bound in the second part of the statement.
\end{proof}

Next, we give a cohomological criterion for the connectedness of the singular scheme of foliations by curves analogous to the criterion given for codimension one distributions in \cite[Theorem 3.8]{CCJ}. 

\begin{Prop}\label{con}
Let $\sF$ be a foliation by curves on $\P^3$ of degree $d\ge2$ with locally free conormal sheaf. If $h^2(N_\sF^*(1-d))=0$, then $Z:=\sing_1(\sF)$ is connected. Otherwise, $Z$ has $h^2(N_\sF^*(1-d))+1$ connected components, when it is reduced.
\end{Prop}

In particular, the singular scheme of a foliation by curves of global complete intersection type is always connected.

\begin{proof}
Taking cohomology on the following exact sequence
$$ 0 \longrightarrow N_\sF^*(1-d)  \stackrel{\phi}{ \longrightarrow} \Omega_{\P^3}^1 (1-d)\to \IZ \longrightarrow 0, $$
we obtain the equality $h^1(\IZ)=h^2(N_\sF^*(1-d))$, since $d\ge2$. It follows that
$$ h^0(\OZ)=h^1(\IZ)+1=h^2(N_\sF^*(1-d))+1. $$ 
If $h^2(N_\sF^*(1-d))=0$, then $Z$ must be connected. Otherwise, if $Z$ is reduced, then the number of connected components of $Z$ is precisely $h^0(\OZ)$.
\end{proof}

\begin{Obs}\rm
The hypothesis $d\ge2$ is necessary. In fact, the conormal sheaf of a foliation by curves of degree 0 is reflexive but not locally free, and its singular set consists of a single point, see the first paragraph of Section \ref{sec:deg12} below. Furthermore, there exist foliations of degree 1 with $N_\sF^*=\O(-2)^{\oplus2}$ whose singular set consists of two skew lines, see Example \ref{Ex-deg1} below. \qed 
\end{Obs}

We complete this section by proving a useful technical result that partially describes the cohomology of the normal sheaves of foliations by curves. 

\begin{Lemma}\label{L:CH}
If $\sF$ is a foliation by curves on $\p3$ of degree $d\ge1$, then
\begin{enumerate}
\item[(i)] $h^0(N_\sF^*(p))=0\mbox{ for } p\leq 1$;
\item[(ii)] $h^1(N_\sF^*(p))=0$ for $p\leq 1-d$;
\item[(iii)] $h^3(N_\sF^*(p))=0$ for $p\geq d-2$.
\end{enumerate}
If, in addition, $\sing_1(\sF)$ is reduced, then $h^2(N_\sF^*(p))=c_3(N_\sF^*)$ for $p\leq -d$.
\end{Lemma}

\begin{proof} 
The  first item follows easily from the exact sequence in display \eqref{pfaff2}, since $h^0(\Omega_{\p3}^1(p))=0$ for $p\leq 1$. Item (iii) is then obtained via Serre duality (see \cite[Thm. 2.5]{H2}), noticing that 
$$ (N_\sF^*)^{*} = N_\sF^*\otimes  \det (N_\sF^*)^{*} = N_\sF^*(d+3) $$
since $N_\sF^*$ is a rank two reflexive sheaf, see \cite[Prop 1.10]{H2}.

For item (ii), we have the exact sequence in cohomology
$$ H^0(\IZ(d+p-1))\rightarrow H^1(N_\sF^*(p))\rightarrow H^1(\Omega_{\p3}^1(p)). $$
The term on the left vanishes for $p+d-1\leq 0$, while the term of the right vanishes for all $p\neq 0$.

When $C=\sing_1(\sF)$ is reduced, we have that $h^0(\OC(k))=0$ for $k\le-1$, thus
$$ h^1(\IZ(k))=h^0(\OZ(k))=h^0(R)=c_3(N_\sF^*) $$
in the same range. The first part of item (v) then follows from the cohomology sequence
\[ H^1(\Omega_{\P^3}^1(p)) \to H^1(\IZ(p+d-1))\rightarrow H^2(N_\sF^*(p))\rightarrow H^2(\Omega_{\P^3}^1(p)), \]
since the leftmost and rightmost terms vanish for $p\le -d<0$.
\end{proof}

If $N_\sF^*$ is locally free, then Serre duality implies that $h^1(N_\sF^*(k))=h^2(N_\sF^*(d-k-1))$, and the following claim follows easily from the previous lemma.

\begin{Cor}
If $\sF$ is a foliation by curves on $\p3$ of degree $d\ge1$ such that $N_\sF^*$ is locally free and $\sing_1(\sF)$ is reduced, then $h^1(N_\sF^*(p))=h^2(N_\sF^*(p))=0$ for $p\geq2d+1$.
\end{Cor}


 Theorem 5.1 in \cite{CJV} states that, under a hypothesis on the codimension of the singular scheme, globally complete intersection foliations of dimension one are determined by their singular schemes. In fact, Lemma \ref{n-loc free} above guarantees that the hypothesis we just mentioned always holds, thus \cite[Theorem 5.1]{CJV} can be restated as follows.

\begin{Theorem}\label{determi}\cite[Theorem 5.1]{CJV}
Let $\mathscr{F}$ be a foliation by curves
on  $\mathbb{P}^m$ of degree $d\geq m-2$ such that $N_\sF^*$ splits as a sum of line bundles. If $\mathscr{F}'$ is a foliation by curves on $\mathbb{P}^m$ of degree $d$, such that $\sing(\mathscr{F}) \subset \sing(\mathscr{F}') $, then $ \mathscr{F}'=\mathscr{F}$.
\end{Theorem}

We finish this section with a technical result that will be useful below.

\begin{Lemma}\label{sub-genusdeg1}
Let $\G$ be a codimension one distribution of degree 1 on $\p3$. If $\sing(\G)$  has degree 3, then it does not contain a double line of genus $<-1$. 
\end{Lemma}
\begin{proof}
Since $\deg(\sing(\G))=3$, then by \cite[section 8]{CCJ} we have that $W:=\sing(\G)$ has pure dimension one. We argue by contradiction and assume that $W$ contains a double line $S$ of genus $p_a(S)<-1=-\deg(\G)$. Then by \cite[Corollary 9]{GJM} the singular scheme  $\sing(\G)$ contains $S_{red}^{(2)}$ the  second infinitesimal neighborhood of  $S_{red}$. In particular,   $\sing(\G)=  S_{red}^{(2)}$, since $\deg(S_{red}^{(2)})=3$.  A contradiction, because by \cite[Section 8, pages 12 and 13]{CCJ}, such distribution would be singular in codimension one.
\end{proof}


\section{Foliations with locally free conormal sheaf of degree 1 and 2} \label{sec:deg12}

Foliations by curves of degree 0 on $\p3$ are quite simple to describe. These are given by the choice of a nontrivial section $\sigma\in H^0(\tp3(-1))$, leading to the exact sequence
$$ \sF ~~:~~ 0\to \op3(1) \stackrel{\sigma}{\to} \tp3 \to S_p \to 0, $$
where $S_p$ is the rank 2 reflexive sheaf defined by the resolution
$$ 0 \to \op3 \to \op3(1)^{\oplus 3} \to S_p \to 0 ,$$
with $p$ being the point where the three linear sections of the first morphism vanish simultaneously. Note that $\sing(\sF)=\{p\}$. Such sheaves are called 1-tail and have been studied, more in general, in \cite{CMM}.

In particular, the conormal sheaf of a foliation by curves of degree 0 is never locally free. The dual description of such foliations is given by the exact sequence
$$ \sF ~~:~~ 0 \to S_p(-3) \to \Omega^1_{\p3} \stackrel{\sigma^\vee}{\to} \mathcal{I}_p(-1) \to 0. $$
 The closure of the leaves of the foliation $\sF$ are lines passing through the singular point $p$. Thus, any two degree zero foliations are equivalent up to automorphism. We refer to \cite[Chapitre 7]{CeD} for the classification of foliations of degree zero and any dimension. 

\bigskip

Let us now consider foliations by curves of degree 1 on $\p3$ with locally free conormal sheaf.

\begin{Theorem} \label{thm-deg1}
If $\sF$ be a foliation by curves on $\P^3$ of degree 1 with locally free conormal sheaf, then $N_\sF^*=\op3(-2)^{\oplus 2}$ and $Sing(\sF)$ consists of two skew lines. 
\end{Theorem}

\begin{proof}
Since $c_1(N_\sF^*)=-4$, \cite[Corollary 2.2]{RV2} tells us that if $N_\sF^*$ does not split as a sum of line bundles, then $h^1(N_\sF^*(1))\ne0$. 

Note from the sequence
$$ 0\to N_\sF^* \to \Omega^1_{\p3} \to \IC \to 0, $$
where $C:=\sing(\sF)$ is a curve of degree $6-c_2(N_\sF^*)$, that $h^1(N_\sF^*(1))=h^0(\IC(1))$, so if $N_\sF^*$ does not split as a sum of line bundles, then $C$ is a plane curve. However, the expressions in display \eqref{eq:gen} imply that $p_a(C)=-1$, and no plane curve can have negative genus.

We conclude that $N_\sF^*$ must split as a sum of line bundles, and the only possibility in degree 1 is $\op3(-2)^{\oplus 2}$. In addition, since $\deg(C)=2$ and $p_a(C)=-1$, $C$ must consist of the union of two skew lines.
\end{proof}

Our next goal is the classification of degree 2 foliations.

\begin{Theorem} \label{thm-deg2}
Let $\sF$ be a foliation by curves on $\P^3$ of degree 2 with locally free conormal sheaf. Then $N_\sF^*=\op3(-2)\oplus\op3(-3)$ and $\sing(\sF)$ is a connected curve of degree 5 and arithmetic genus 1. 
\end{Theorem}
\begin{proof}
Lemma \ref{bound c2} tells us that $4\le c_2(N_\sF^*)\le 9$, while the expressions in displays \eqref{eq:degZ} and \eqref{eq:gen} yield
$$ \deg(C) = 11 - c_2(N_\sF^*) ~~{\rm and} $$
$$ p_a(C) = 10 - 3c_2(N_\sF^*)/2, $$
where $C:=\sing(\sF)$. Since $c_2(N_\sF^*)$ must be even, there are only 3 possible values: 4, 6 and 8.

Set $E:=N_\sF^*(2)$, and note that $c_1(E)=-1$ and $c_2(E)=c_2(N_\sF^*)-6$. Lemma \ref{L:CH} implies that $h^0(E(k))=0$ for $k\le-1$.  

If $c_2(N_\sF^*)=4$, then $c_1(N_\sF^*)^2-4c_2(N_\sF^*)=9>0$, so $N_\sF^*$ cannot be $\mu$-stable. It follows that $h^0(E)\ne0$, so let $\sigma\in H^0(E)$ be a nontrivial section. If $\sigma$ does not vanish, then $E=\op3\oplus\op3(-1)$,
which contradicts $c_2(E)=-2$; so let $Y:=(\sigma)_0$. It would follow that $\deg(Y)=c_2(E)=-2$, again a contradiction.

Now let $c_2(N_\sF^*)=6$. Again $c_1(N_\sF^*)^2-4c_2(N_\sF^*)=1>0$, so $N_\sF^*$ cannot be $\mu$-stable. Again, it follows that $h^0(E)\ne0$, but since $c_2(E)=0$ any section $\sigma\in H^0(E)$ must be nowhere vanishing, thus $E=\op3\oplus\op3(-1)$, implying that $N_\sF^*=\op3(-2)\oplus\op3(-3)$. The connectedness of $\sing(\sF)$, in this case, is an immediate consequence of Proposition \ref{con}.

Finally, assume that $c_2(N_\sF^*)=8$, so $c_2(E)=2$. If $h^0(E)\ne 0$, then we obtain the diagram
$$\xymatrix{
& 0\ar[d] & 0\ar[d] & & \\
& \op3 \ar@{=}[r]\ar[d] & \op3\ar[d] & & \\
0\ar[r] &  E \ar[r]\ar[d] & \Omega^1_{\p3}(2) \ar[r]\ar[d] & \IC(3) \ar[r]\ar@{=}[d] & 0 \\
0\ar[r] & \mathcal{I}_Y(-1) \ar[r]\ar[d] & N(1) \ar[r]\ar[d] & \IC(3) \ar[r] & 0 \\
& 0 & 0 & &
}$$
where $N$ is a non locally free null-correlation sheaf, and $Y:=(\sigma)_0$ is a curve of degree $c_2(E)=2$; in addition, $C$ is a curve of degree 3 and genus -2. However, we claim that there can be no injective morphism $\mathcal{I}_Y(-1)\hookrightarrow N(1)$ when $\deg(Y)>1$, leading to a contradiction. Indeed, a non locally free null-correlation sheaf satisfies the following short exact sequence
$$ 0 \to N \to \op3^{\oplus2} \to \mathcal{O}_L(1) \to 0, $$
where $L$ is a line \cite[Section 6]{JMT}. A monomorphism $\varphi:\mathcal{I}_Y(-1)\hookrightarrow N(1)$ would then lead to the following commutative diagram
$$\xymatrix{
& 0\ar[d] & 0\ar[d] & 0\ar[d] & \\
0\ar[r] & \mathcal{I}_Y(-1) \ar[r]\ar[d]^{\varphi} & \op3(-1) \ar[d]^{\varphi^{\vee\vee}} \ar[r] & \mathcal{O}_Y(-1) \ar[r] \ar[d] & 0 \\
0\ar[r] &  N(1) \ar[r] & \op3(1)^{\oplus2} \ar[r] & \mathcal{O}_L(1) \ar[r] & 0 
}$$
implying that $Y$ must be a subscheme of the line $L$, thus $\deg(Y)\le1$.

It follows that $E$ must be a $\mu$-stable rank 2 bundle with $(c_1(E),c_2(E))=(-1,2)$.   In particular, we have that $H^0(E(1))\neq 0$ \cite[Proposition 1.1]{HS}. Then, we have the following diagram
$$ 
\xymatrix{
&0 \ar[d] &0 \ar[d] \\
& \op3  \ar[d] \ar@{=}[r]&\op3  \ar[d] \\
0 \ar[r]  &  E(1) \ar[r]\ar[d] & \Omega_{\mathbb{P}^3}^1(3)\ar[r] \ar[d]  &  \mathcal{I}_C(4) \ar@{=}[d]\ar[r]&  0 & \\
 0 \ar[r]  & \mathcal{I}_S(1)\ar[r] \ar[d]  & G  \ar[d]\ar[r]  &  \mathcal{I}_C(4)   \ar[r]&  0 &
 \\
  &0 &0 &&,
}
$$
where $S$ is double line of genus $-2$, so that $\omega_S(3)\simeq \mathcal{O}_S$. Dualizing the middle row we get
$$
0 \longrightarrow  G^* \longrightarrow  T\p3(-3) \longrightarrow  \mathcal{O}_{\p3}  \longrightarrow  \inext^1(G,\op3) \longrightarrow 0;
$$
cutting it into short exact sequences, we conclude that $\inext^1(G,\op3)\simeq \mathcal{O}_W$ and obtain the following codimension one distribution of degree 1
$$
\mathscr{D}: 0 \longrightarrow  G^*(3) \longrightarrow  T\p3 \longrightarrow  \mathcal{I}_{W}(3)  \longrightarrow 0,
$$
where $W=\sing(\mathscr{D})$. 

Dualizing the bottom line of the above  diagram  we obtain the epimorphism
$$ \inext^1(G,\op3)\simeq \mathcal{O}_W \twoheadrightarrow
\inext^1(\mathcal{I}_S(1),\op3) \simeq \mathcal{O}_S $$
from which it follows that $S\subset W$.

Let $W_1$ denote the component of pure  dimension 1 of $W$, so that $\deg(W_1)\geq 2$.
According to the classification of codimension one distributions of degree 1 (see \cite[Section 8]{CCJ}),   we have the following possibilities:
\begin{itemize}
\item $W_1$ is a  (possibly degenerate) conic; in this case, we must have that $S=W_1$, but this leads into a contradiction, since $S$ is a double line of genus $-2$. 

\item  $W=W_1$  is a curve of  degree $3$   and  $S\subset W$  is  a double line of genus $p_a(S)=-2<-1$. But this contradicts Lemma \ref{sub-genusdeg1}. 
\end{itemize}
 
\end{proof}

\begin{Cor}
Every foliation by curves with locally free conormal sheaf and degree 1 or 2 is contained in a pencil of planes or is Legendrian, and is given by the global  complete intersection of two codimension one distributions.

 In addition, the moduli space of 
locally complete intersection foliations 
of degree $d\in\{1,2\}$ is an irreducible quasi-projective variety of dimension 8 if $d=1$, and dimension $20$  if $d=2$.
\end{Cor}
\begin{proof}
It follows from Theorems \ref{thm-deg1} and \ref{thm-deg2} that 
the  foliation $\sF$ is such that   $N_\sF^*$ is split and $\op3(-2) \subset N_\sF^*\subset \Omega_{\mathbb{P}^3}^1$ is a degree zero  codimension one distribution which  contains $\sF$. Now, the result follows from the classification codimension one distribution  of degree zero \cite[Proposition 7.1]{CCJ}.

Since a generic point of the moduli space of 
locally complete intersection foliations 
of degree $d\in\{1,2\}$ is a Legendrian foliation, then  the second part of the Corollary follows from  Theorem \ref{thm-leg}. See also remark \ref{Obs-moduli-split}. 

\end{proof}

 Item (3) of Main Theorem  \ref{classification-degree12}
 follows from  Theorem   \ref{determi}.

\begin{Ex}\label{Ex-deg1}\rm 
Let $\sF_1$ and $\sF_2$ be one-codimensional distributions on
$\mathbb{P}^3$ induced, respectively, by    the $1$-forms
$\omega_1=x_{0}dx_1-x_{1}dx_{0}$ and
$\omega_2=x_{0}dx_1-x_{1}dx_{0}+ x_{2}dx_3-x_{3}dx_{2}$. We have that the complete intersection $\sF=\sF_1\cap\sF_2$, of degree one,  is given by
$$
\omega=\omega_1\wedge\omega_2=x_0x_2dx_1 \wedge dx_3- x_0x_3dx_1
\wedge dx_2 -x_1x_2dx_0\wedge  dx_3+x_1x_3dx_0\wedge dx_2
$$
with $\mathrm{Sing}(\sF)=\{x_0=x_1=0\}\cup\{ x_2=x_3=0\}$.
Since $\omega_2$ induces a  contact distribution and $\omega_1$ induces a pencil of planes on $\mathbb{P}^3$, we have that $\sF$ is a Legendrian foliation whose leaves are  contained in a pencil of planes. 
\end{Ex}

\begin{Ex}\label{Ex-deg2}\rm 
Consider the five lines $L_1=\{x_{2}=x_{1}=0\}$, $L_2=\{x_{2}-x_{3}=x_{0}-x_{1}=0\}$,
$L_3=\{x_{1}=x_{0}=0\}$, $L_4=\{x_{1}-x_{3}=x_{0}-x_{2}=0\}$ and $L_5=\{x_{1}-x_{2}=x_{0}-x_{3}=0\}$; and set $C= \cup_{i=1}^5L_i$.
Then, $C$ is a curve of degree $5$ and  arithmetic genus $1$.  
The foliation  by curves $\sF$  of degree $2$  induced  by the quadratic vector field 
$$
v=(3\,x_{0}^{2}-5\,x_{1}^{2}-7\,x_{0}x_{2}+5\,x_{1}x_{2}-x_{0}x_{3}+5\,x_{1}x_{3})\frac{\partial}{\partial x_0}-(2\,x_{0}x_{1}+2\,x_{1}x_{2}-4\,x_{1}x_{3})\frac{\partial}{\partial x_1}+
$$
$$
+(5\,x_{0}x_{1}-5\,x_{1}^{2}-7\,x_{0}x_{2}+3\,x_{2}^{2}+5\,x_{1}x_{3}-x_{2}x_{3})\frac{\partial}{\partial x_2}-
(5\,x_{0}x_{1}-5\,x_{1}^{2}+5\,x_{1}z_{2}-3\,x_{0}x_{3}-3\,x_{2}x_{3}+x_{3}^{2})\frac{\partial}{\partial x_3}
$$
is such that $\mathrm{Sing}(\sF)=C$. It follows from Theorem \ref{thm-deg2} that $\sF$ is the only foliation, of degree $2$, singular on $C$ and it   is a global complete intersection of
two codimension one distributions. 
\end{Ex}


\section{Foliations with locally free conormal sheaf of degree 3} \label{sec:deg3}

In this section, we will prove the classification of the topological and algebraic invariants of foliations by curves of degree 3 with locally free conormal sheaf stated in Main Theorem \ref{classification-degree3}. Note that a foliation $\sF$ of degree $3$ is given by the short exact sequence
$$ 0 \to N_\sF^* \to \Omega_{\mathbb{P}^3}^1 \to \IC(2) \to  0. $$

We begin by considering global  complete intersection foliations, that is, the case when the conormal sheaf $N_\sF^*$ splits as a sum of line bundles. These correspond to cases (1) and (2) of Main Theorem \ref{classification-degree3}.

Since $c_1(N_\sF^*)=-6$, it is easy to see that $\op3(-2)\oplus\op3(-4)$ and $\op3(-3)^{\oplus2}$ are the only possibilities in degree 3. The connectedness of the singular scheme is a straightforward consequence of Proposition \ref{con}, while the   calculation of its degree and genus uses the formulas in display \eqref{eq-deggenus}.

To start addressing the third item of Main Theorem \ref{classification-degree3}, we establish the following result.

\begin{Prop}\label{prop-stable}
Let $\sF$ be a foliation by curves of degree $3$ on $\p3$ of local complete intersection type. If $N_\sF^*$ does not split as a sum of line bundles, then it is stable and
$10\le c_2(N_\sF^*)\le 16$.
\end{Prop}

\begin{proof}
We start by showing that $h^0(N_\sF^*(2))=0$. Suppose that $h^0(N_\sF^*(2))\neq 0$; a nontrivial section $\sigma\in H^0(N_\sF^*(2))$ induces a codimension one distribution of degree 0 on $\p3$
$$ \G: 0 \to \O(-2) \to \Omega^1_{\p3} \to \Omega_{\G} \to 0 $$
which contains $\sF$.  Indeed, having $H^0(N_\sF^*(1))=0$ by Lemma \ref{L:CH}-(i), $\sigma$ cannot vanish in codimension 1. Then, by the classification of such distributions (see \cite[Proposition 7.1]{CCJ}), either $\G$ is the non-singular contact distribution or $\G$ is a pencil of planes. In the first case, the section $\sigma:\O(-2)\to N_\sF^*$ cannot vanish, which implies that $N_\sF^*$ must split as a sum of line bundles, contradicting our hypothesis. In the second case, the zero locus of $\sigma$ is a line since this is the singular set of the pencil of planes $\G$. We, therefore, have an exact sequence of the form
$$ 0\to \op3(-2) \to N_\sF^* \to \mathcal{I}_L( -4) \to 0 $$
where $L$ is a line in $\p3$. However, $L$ is a complete intersection curve, so \cite[Corollary 1.2]{H1} implies that $N_\sF^*$ would again  split as a sum of line bundles.


Notice that the stability of $N_\sF^*$ is equivalent to $h^0(N_\sF^*(3))=0$, since $N_\sF^*(3)$ is the normalization of $N_\sF^*$.

Let us suppose that, on the contrary, $N_\sF^*(3)$ admits a non-trivial global section, which we also denote by $\sigma$. Since $h^0(N_\sF^*(2))=0$, its cokernel must be a torsion free sheaf of rank 1, so it must be the twisted ideal sheaf of a curve $S$ in $\p3$, that is $\coker\sigma\simeq \mathcal{I}_S( -3)$.

In addition, $\sigma$ induces the commutative diagram in display \eqref{diag-xx} below, with the middle column being a codimension one distribution $\G$ of degree 1.

\begin{equation}\label{diag-xx}
\begin{split} \xymatrix{
& 0 \ar[d]& 0 \ar[d]\\
0 \ar[r] & \O(-3) \ar[r]^{\simeq} \ar[d]^\sigma & \O(-3) \ar[d] & & &  \\
0 \ar[r] & N_\sF^* \ar[r] \ar[d] & \Omega^1_{\p3} \ar[r] \ar[d] & \IC(2) \ar[r] \ar[d]^{\simeq} & 0\\
0\ar[r] & \mathcal{I}_S( -3 ) \ar[d] \ar[r] & \Omega_{\G} \ar[r] \ar[d] & \IC(2) \ar[r] & 0 \\
& 0 & 0 & 
} \end{split} \end{equation}


It follows from the bottom row of diagram \eqref{diag-xx} that $S \subset \sing(\G)$.  
 Indeed, being $N_\sF^*$ locally free implies, by Lemma \ref{n-loc free}, that the curve $C$ in the diagram is of pure dimension 1. Therefore, applying the functor $\mathcal{H}om(\cdot,\O)$ to the bottom row of this diagram, we have that $\mathcal{E}xt^2(\IC(2),\O) =0$ and an epimorphism
$$
\mathcal{E}xt^1(\Omega_{\G},\O) \simeq \mathcal{O}_W(3) \twoheadrightarrow \mathcal{E}xt^1(\mathcal{I}_S(-3),\O)  \simeq \mathcal{O}_S(3)
$$
implies the desired inclusion $S \subset W:=\sing(\G)$.  Let $S'$ denote the maximal 1-dimensional subscheme of $\sing(\G)$. According to the classification of codimension one distributions of degree 1  \cite[Section 8]{CCJ}, $S'$ is a curve of degree at most 3 with $p_a(S')=0$. It follows that $1\le\deg(S)\le3$, and $p_a(S)=1-2 \deg(S)$ \cite[Proposition 2.1]{H1}. Let us now look at each of the three possible cases.

\begin{itemize}
\item If $\deg(S)=1$, then $S$ is a line, contradicting $p_a(S)=-1$.

\item 
 If $\deg(S)=2$, then $p_a(S)=-3$, so $S$ is a double line. 
Since $S \subset S'$, then either $\deg(S')=2$ or $\deg(S')=3$. In the first case, we would have $S=S'$, but $p_a(S)=-3$ while $p_a(S')=0$, a contradiction. The second case contradicts Lemma \ref{sub-genusdeg1}.



\item If $\deg(S) =3$, then $S=S'$ but $p_a(S)=-5$ while $p_a(S')=0$.
\end{itemize}

The lower bound on $c_2(N_\sF^*)$ is a direct consequence of the fact that $N_\sF^*$ is stable via the Bogomolov inequality, see \cite[Lemma 3.2]{H1}. The upper bound is the one given in Lemma \ref{bound c2}.

\end{proof}

The next step is to rule out the possibilities $c_2(N_\sF^*)=16$ and $c_2(N_\sF^*)=15$, which requires a detailed analysis of the possible conormal sheaves and singular schemes. For this purpose, recall that every rank 2 locally free sheaf $F$ on $\p3$ with $c_1(F)=0$ is isomorphic to  the cohomology of a monad of the form
\begin{equation}\label{gen monad}
\bigoplus_{i=1}^s \op3(-c_i) \to
\bigoplus_{j=1}^{s+1} \op3(-b_j)\oplus\op3(b_j) \to
\bigoplus_{i=1}^s \op3(c_i),
\end{equation}
with $1\le c_1\le\dots\le c_s$ and $0\le b_1 \le\dots\le b_{s+1}$. Furthermore, recall also that a sheaf $F$ on $\mathbb{P}^n$ is $k$-regular (in the sense of Castelnuovo--Mumford) if $h^p(F(k-p))=0$ for every $p>0$; moreover, every $k$-regular sheaf is also $k'$-regular for every $k'\ge k$.

With these facts in mind, we now state a useful technical result.

\begin{Lemma}\label{l:reg}
Let $F$ be a stable rank 2 locally free sheaf on $\p3$ with $c_1(F)=0$.
\begin{enumerate}
\item If $c_2(F)=7$, then $F$ is 13-regular.
\item If $c_2(F)=6$, then $F$ is 10-regular. Furthermore, if $F$ is not isomorphic to the cohomology of one of the following two monads
\begin{equation}\label{bad monad 1}
 \op3(-2)^{ \oplus 2}\oplus\op3(-1) \to 
 \op3(-1)^{ \oplus 3}\oplus \op3^{ \oplus 4}\oplus  \op3(1)^{ \oplus 3} \to
\op3(1)\oplus \op3(2)^{ \oplus 2}
\end{equation}
\begin{equation}\label{bad monad 2}
\op3(-3)\oplus\op3(-1) \to 
\op3(-2)\oplus \op3^{ \oplus 4}\oplus\op3(2) \to
\op3(1)\oplus\op3(3)
\end{equation}
then $F$ is 8-regular.
\end{enumerate}
\end{Lemma}

\begin{proof}
According to \cite[Theorem 3.2]{CMR}, the cohomology of a monad of the form \eqref{gen monad} is at least $k$-regular with
\begin{equation}\label{form reg}
k=2c_s+b_3+\cdots+b_{s+1}+c_1+\cdots+c_s-2.
\end{equation}
On the other hand, Hartshorne and Rao classify in \cite[Section 5.3]{HR} all possible monads for stable rank 2 bundles $F$ on $\p3$ with $c_1(F)=0$ and $c_2(F)\le 8$. The claims in the statement of the lemma are obtained simply by applying the formula in display \eqref{form reg} to all monads listed in \cite[Section 5.3]{HR}.
\end{proof}

In order to use the previous result, we introduce the notation $E_\sF:=N_\sF^*(3)$ for any foliation $\sF$ by curves of degree 3 of local complete intersection type. Note that $c_1(E_\sF)=0$, so $E_\sF$ is the normalization of the conormal sheaf $N_\sF^*$; in addition, we have $c_2(E_\sF)=c_2(N_\sF^*)-9$ and 
\begin{equation}\label{serre}
h^1(E_\sF(k)) = h^2(E_\sF(-4-k))=h^2(N_\sF^*(-k-1))=h^1(\IC(1-k))
\end{equation}
by Serre duality, where $C=\sing(\sF)$.

\begin{Prop}
There are no foliations by curves of local complete intersection type and degree 3 on $\p3$ with $c_2(N_\sF^*)=16$.
\end{Prop}
\begin{proof}
Let $\sF$ be a foliation as in the statement of the lemma. According to the formulas in displays \eqref{eq:degZ} and \eqref{eq:gen}, the singular locus of $\sF$ is a curve $C$ of degree 2 and genus $-13$; let $L:=C_{\rm red}$, so that $C$ is a multiplicity 2 structure on the line $L$. Following \cite[Proposition 1.4 and Remark 1.5]{N}, we must have the exact sequence
$$ 0\to \mathcal{O}_L(12) \to \mathcal{O}_C \to \mathcal{O}_L \to 0, $$
thus, by the equalities in display \eqref{serre}, $h^1(E_\sF(13))=h^1(\IC(-12))=1$, meaning that $E_\sF$ is not 12-regular, thus contradicting the first part of Lemma \ref{l:reg} since $c_2(E_\sF)=7$.
\end{proof}

Let us now shift our attention to the case $c_2(N_\sF^*)=15$.

\begin{Lemma}\label{triple line}
Let $\sF$ be a foliation by curves of local complete intersection type and degree 3 on $\p3$. If $c_2(N_\sF^*)=15$, then $C$ is a multiplicity 3 structure on a line $L$ satisfying the exact sequence
$$ 0\to \mathcal{O}_L(a)\oplus\mathcal{O}_L(c) \to \mathcal{O}_C \to \mathcal{O}_L \to 0, $$
with $(a,c)$ equal either to $(1,7)$ or to $(2,6)$. 
\end{Lemma}

\begin{proof}
Let $\sF$ be a foliation by curves of local complete intersection type and degree 3 on $\p3$ with $c_2(N_\sF^*)=15$, so that $c_2(E_\sF)=6$. It follows from the formulas in displays \eqref{eq:degZ} and \eqref{eq:gen} that the singular locus of $\sF$ is a curve $C$ of degree 3 and genus $-10$. Letting $L:=C_{\rm red}$, three possibilities may occur:
\begin{itemize}
\item $L$ is the disjoint union of skew lines $L_1\sqcup L_2$, so that $C$ is the disjoint union of a multiplicity 2 structure on the line $L_1$, say, of genus $-9$ with a line. We, therefore, have an exact sequence
$$ 0 \to \mathcal{O}_{L_1}(8) \to \mathcal{O}_C \to \mathcal{O}_{L} \to 0, $$
so that $h^1(\IC(-8))=1$. It follows from the equality in display \eqref{serre} that $h^1(E_\sF(9))=1$, so $E_\sF$ is not 10-regular, contradicting the second part of Lemma \ref{l:reg}. 
\item $L$ is the union of intersecting lines $L_1\cup L_2$, with $L_1$ carrying a multiplicity 2 structure in $C$. This time we have the exact sequence
$$ 0 \to \mathcal{O}_{L_1}(9) \to \mathcal{O}_C \to \mathcal{O}_{L} \to 0, $$
so that $h^1(E_\sF(10))=h^1(\IC(-9))=1$, again contradicting the second part of Lemma \ref{l:reg}.
\item $L$ is a line. According to \cite[Section 2]{N}, multiplicity 3 structures of genus $-10$ on a line satisfy the exact sequence
$$ 0 \to \mathcal{O}_{L}(a)\oplus\mathcal{O}_{L}(2a+b) \to \mathcal{O}_C \to \mathcal{O}_{L} \to 0, $$
where $a\ge-1$, $b\ge0$ and $3a+b+2=10$. We, therefore, have four possibilities:
$$ (a,b)=(-1,11),~(0,8),~(1,5),~{\rm and}~(2,2). $$
In the first two cases we would have $h^1(E_\sF(9))=h^1(\IC(-8))=h^0(\mathcal{O}_{L}(2a+b-8)\ne0$, so $E_\sF$ is not 10-regular, again contradicting the second part of Lemma \ref{l:reg}. Setting $c:=2a+b$, we obtain the desired result.
\end{itemize}
\end{proof}

In particular, note that $h^1(E_\sF(7))\ne0$, so if $\sF$ is a foliation by curves of local complete intersection type and degree 3 with $c_2(N_\sF^*)=15$, then $h^1(E_\sF(7))=h^1(\IC(-6))\ne0$, so $E_\sF$ is not 8-regular. Therefore, $E_\sF$ must be the cohomology of a monad either as in display \eqref{bad monad 1} or as in display \eqref{bad monad 2}. This is the starting point of the proof of our last result in this section.

\begin{Prop}
There are no foliations by curves of local complete intersection type and degree 3 on $\p3$ with $c_2(N_\sF^*)=15$.
\end{Prop}
\begin{proof}
Let $\sF$ be a foliation as in the statement of the lemma, and set $E_\sF=N_\sF^*(3)$. We have that $h^3(E_\sF(1))=h^0(E_\sF(-3))=0$, thus, using Grothendieck--Riemann--Roch Theorem to compute $\chi(E_\sF(1))$, we also have
$$ h^1(E_\sF(1)) = h^0(E_\sF(1)) + h^2(E_\sF(1)) - \chi(E_\sF(1)) = 10 + h^0(E_\sF(1)) + h^2(E_\sF(1)). $$
As observed above, $E_\sF$ must be the cohomology of monad either as in display \eqref{bad monad 1} or as in display \eqref{bad monad 2}; in both case, $h^0(E_\sF(1))\ne0$, see \cite[Section 5.3]{HR}. It follows, using the equality in display \eqref{serre}, that $h^1(\IC)=h^1(E_\sF(1))>10$. 

On the other hand, let us examine the exact sequence
$$ 0 \to \IC \to \mathcal{I}_L \to \mathcal{O}_L(a)\oplus\mathcal{O}_L(c) \to 0 ,$$
which is equivalent to the sequence in the statement of Lemma \ref{triple line}, with $(a,c)=(1,7),(2,6)$. Since $h^1(\mathcal{I}_L)=0$, we have that
$$ h^1(\IC) = h^0(\mathcal{O}_L(a)) + h^0(\mathcal{O}_L(c)) = a+c+2 = 10, $$
providing the desired contradiction.
\end{proof}

We have so far proved the first part of Main Theorem \ref{classification-degree3}, concerning items (1), (2), and (3). 


\subsection{Foliations with reduced singular scheme}

We now move to the second part of Main Theorem \ref{classification-degree3}, making the further assumption that $\sing(\sF)$ is reduced.

 Recall that, by Proposition \ref{prop-stable}, if $N_\sF^*$ does not split as a sum of line bundles, then it is $\mu$-stable. Since a complete description of the split case has already been given in the beginning of this section, we will now assume that $N_\sF^*$ is $\mu$-stable. 

Recall furthermore that a locally free sheaf is said to have \emph{natural cohomology} if for each $p\in\Z$ there can be at most one $i=0,1,2,3$ such that $h^i(E(p))\ne0$ (see for example \cite{HH}); note that if $E$ is stable and has rank 2, then $\chi(E(-2))=h^2(E(-2))-h^1(E(-2))=0$, so stable rank 2 locally free sheaves with natural cohomology are necessarily instanton bundles. However, not every instanton bundle has natural cohomology ('t Hooft instanton bundle of charge at least 3 are the most well-known exceptions). In any case, instanton bundles with natural cohomology form an open subset in the moduli space of instanton bundles.

To be precise, we establish the following result. 
 
\begin{Prop}\label{exclusions}
Let $\sF$ be a foliation of local complete intersection type and degree 3 such that $\sing(\sF)$ is reduced and $N_\sF^*$ is $\mu$-stable. Then
\begin{itemize}
\item we cannot have that $c_2(N_\sF^*)=14$;
\item if $c_2(N_\sF^*)=13$ then $E_\sF=N_\sF^*(3)$ is an instanton bundle of charge 4 with natural cohomology;
\item if $c_2(N_\sF^*)=12$ then $E_\sF=N_\sF^*(3)$ is an instanton bundle of charge 3 with $h^0(E(1))\le1$.
\end{itemize}
\end{Prop} 

 As stated in the Introduction, a locally complete intersection foliation by curves that is not a globally complete intersection may not be uniquely determined by its singular scheme. This phenomenon is illustrated by the example below, and it also shows that our classification of locally complete intersection foliation by curves of degree 3, provided in Main Theorem 2, is only a \textit{topological one}.
 
\begin{Ex} \label{example5lines}\rm
Let $C=L_1 \sqcup L_2 \sqcup L_3 \sqcup  L_4 \sqcup  L_5 $ be a disjoint union of five lines in $\p3$ which have no 5-secant line, that is, no line cuts all five lines in $C$.  Denote by $v_{ij}:\O \to T\p3 $ the foliation by curves  such that 
$\sing(v_{ij})= L_i \sqcup L_j$. 
For $1\leq i<j \leq \ell \leq 5$, let be
$Q_{ij\ell}=\{q_{ij\ell}=0\}$
the quadric surface containing $L_i \sqcup L_j \sqcup L_\ell$. 
In  \cite[Lemma 2]{ACM} the authors
show that every member of the following 2-dimensional family of degree 3 foliations by curves
$$
a_0q_{345}\cdot v_{12} + a_1q_{145}\cdot v_{23} + a_2q_{125}\cdot v_{34}: \O(-2) \to T\p3,
$$
where $(a_0:a_1:a_2)$ is a generic point in $\p2$, are singular on $C$. 
\end{Ex} 


Letting $E_\sF:=N_\sF^*(3)$, a foliation by curves of local complete intersection type and degree 3 can be described by the following  short exact sequence
$$ 0 \longrightarrow E_\sF(-3) \longrightarrow \Omega_{\p3}^1 \longrightarrow \mathcal{I}_C(2) \longrightarrow 0,
$$
with $C$ being a curve. Observe that $h^3(E_\sF(1))=h^0(E_\sF(-5))=0$ since $E_\sF$ is $\mu$-stable and $c_1(E_\sF)=0$, while $h^2(E_\sF(1))=h^1(E_\sF(-5))=0$ since $c_2(E_\sF)\le5$. It follows that
$$ \chi(E_\sF(1)) = h^0(E_\sF(1))-h^1(E_\sF(1)) = 8 - 3c_2(E_\sF). $$
Since $h^2(N_\sF^*(-2))=h^1(E_\sF(1))$, the argument in the proof of Proposition \ref{con}, yields
\begin{equation}\label{h0(OZ)}
h^0(\OZ) = 1 + h^1(E_\sF(1)) = 3c_2(E_\sF) - 7 + h^0(E_\sF(1)).
\end{equation}
This is the key fact to be explored in the proof of Proposition \ref{exclusions}.

We will also require the following additional fact.

\begin{Lemma}\label{lemma-disjoint}
Let $C$ be the disjoint union of five lines in $\p3$. Then, there exists an epimorphism
\begin{equation}\label{eq-surj-5disj}
\Omega_{\p3}^1 \stackrel{\varpi}{\longrightarrow} \mathcal{I}_C(2)
\end{equation}
if and only if $C$ has no 5-secant line. Furthermore, $\ker\varpi(3)$ is  an instanton bundle of charge 4 with natural cohomology.
\end{Lemma}

\begin{proof}
If $C$ has no 5-secant line, then the result is proved in \cite[Lemma 2]{ACM}.

Let us suppose the existence of a 5-secant line $L$. Directly from the canonical exact sequence of the canonical sheaf of $C$, we have, restricting on $L$, a surjective map
$$
(\mathcal{I}_C)_{|L} \longrightarrow \mathcal{O}_L(-5).
$$
If a surjective map as in (\ref{eq-surj-5disj}) exists, then, restricting to the line $L$, we would obtain the following composition, which is again surjective
$$
(\Omega_{\P^3}^1)_{|L} \longrightarrow (\mathcal{I}_C(2))_{|L} \longrightarrow \mathcal{O}_L(-3).
$$
Being $(\Omega_{\P^3}^1)_{|L} \simeq \O(-2) \oplus \O(-1)^2$, we get a contradiction. 

Now let $F:=\ker\varpi(3)$, and note that $c_1(F)=0$, $c_2(F)=4$ and $c_3(F)=0$; in other words,
$$ 0 \to F(-3) \to \Omega_{\p3}^1 \stackrel{\varpi}{\to} \mathcal{I}_C(2) \to 0, $$
is a foliation by curves of local complete intersection type and degree 3. Since $Z$ is not an ACM curve, $F$ does not split as a sum of line bundles, and therefore, by Proposition \ref{prop-stable}, $F$ must be stable. As it was observed in \cite[Lemma 5]{ACM}, the fact that $C$ does not have a 5-secant line implies that $h^i(\mathcal{I}_C(3))=0$ for $i=0,1$. It follows that $h^1(F(-2))=0$, forcing $F$ to be an instanton bundle; in addition, we have that 
$$ h^1(F(2))=h^2(F(-6))=h^1(\IZ(-3))=0, $$
thus $F$ has natural cohomology, since $h^1(F(t))=0$ for $t\ge3$ as every instanton bundle of charge 4 is 4-regular.
\end{proof}

We are finally in the position to complete the proof of Proposition \ref{exclusions}. We go over each item separately.

First, if $c_2(N_\sF^*)=14$, we
have from the second item in Theorem \ref{P:SLocus} that $\deg(C)=4$ while $h^0(\mathcal{O}_C)\geq8$ by the formula in display \eqref{h0(OZ)}; but this is impossible for a reduced curve.

Similarly, if $c_2(N_\sF^*)=13$, then we have that $\deg(C) = 5$ and $p_a(C)=-4$, while $h^0(\mathcal{O}_C)=5+h^0(E_\sF(1))\ge 5$. If $C$ is reduced, we must have that $Z$ consists of 5 skew lines; Lemma \ref{lemma-disjoint} then implies that $E_\sF$ must be an instanton bundle of charge 4 with natural cohomology.

Finally, if $c_2(N_\sF^*)=12$, then we have that $\deg(C) = 6$ and $p_a(C)=-1$, while $h^0(\mathcal{O}_C)=2+h^0(E_\sF(1))$. Note that a curve of degree 6 with 4 connected components, must either be the disjoint union of two conics and two lines, or the disjoint union of one cubic and three lines. Since none has arithmetic genus equal to $-1$, we conclude that $h^0(E_\sF(1))\le1$. This restriction not only rules out $E_\sF$ being a special 't Hooft instantons of charge 3, but also the generalized null-correlation bundle given as the cohomology of a monad of the form
$$ \op3(-2) \to \op3(-1)\oplus  \op3^{\oplus2}\oplus\op3(1) \to \op3(2). $$
Since, according to the classification by Hartshorne and Rao \cite[Table 5.3]{HR}, a stable rank 2 bundle $E$ with $c_1(E)=0$ and $c_1(E)=3$ is either an instanton bundle or a generalized null-correlation bundle as above, this completes the proof of Proposition \ref{exclusions}.

We complete this section by observing that it is also possible to discard some of the possible foliations by curves of local complete intersection type and degree 3 with $c_2(N_\sF^*)=13$ or 14 without the hypothesis of $\sing(\sF)$ being reduced. Here are two cases.

\begin{Lemma}
Let $E$ be an instanton bundle of charge 5 with natural cohomology. Then there are no foliations by curves $\sF$ such that $N_\sF^*=E(-3)$.
\end{Lemma}
\begin{proof}
If $E$ is a stable rank 2 locally free sheaf with $c_2(E)=5$, then $$\chi(E(2))=h^0(E(2))-h^1(E(2))=0.$$ If $E$ has natural cohomology, then $h^0(E(2))=h^1(E(2))=0$. However, $E\otimes\Omega^1_{\p3}(3)$ is a subsheaf of $E(2)^{\oplus 4}$, hence $\Hom(E(-3),\Omega^1_{\p3})=0$, so there can be no foliation by curves with $N_\sF^*=E(-3)$.
\end{proof}

\begin{Lemma}
Let $E$ be a 't Hooft instanton bundle of charge  4 or 5. Then there are no foliations by curves $\sF$ such that $N_\sF^*=E(-3)$.
\end{Lemma}
\begin{proof}
If $E$ be a 't Hooft instanton bundle of charge $n$, then $E(1)$ has a global section that vanishes along $n+1$ skew lines. We will discuss the case $n=4$ in detail; the case $n=5$ can be dealt with similarly (the argument is even simpler).

Assume that 
$$ 0 \to E(-3) \stackrel{\varphi}{\to} \Omega^1_{\p3} \to \IC(2) \to 0 $$
defines a foliation by curves, and let  $\sigma\in H^0(E(1))$ be a nontrivial global section. The composition of monomorphisms $$ \op3(-4) \stackrel{\sigma}{\to} E(-3) \stackrel{\varphi}{\to} \Omega^1_{\p3} $$
induces a codimension 1  distribution  of degree 2
$$ 0 \to G^* \to T\p3 \to \mathcal{I}_W(4) \to 0, $$
where $G:=\coker(\varphi\circ\sigma)$, and $W$ is its singular scheme. Setting $C_0:=(\sigma)_0$ (in the case at hand, $C_0$ consists of 5 skew lines), observe that $C_0\subseteq W$. According to the classification of codimension one distributions of degree 2 studied in \cite[Section 9]{CCJ}, we know that $\deg(W)\le7$, so there are 3 possibilities to be considered.

First, if $\deg(W)=5$, then actually $W=C_0$; however, either $p_a(W)=1$ or $p_a(W)=2$, contradicting that $p_a(C_0)=-4$.

If $\deg(W)=6$, then $p_a(W)=3$ (cf. \cite[Theorem 9.5]{CCJ}), and one must consider two possibilities. First, assume that $W$ is reduced, so that $W=C_0\cup L$, where $L$ is a line, implying that $p_a(W)=k-5\le0$, where $k$ is the number of points in $C_0\cap L$ (note that $0\ge k\ge5$). If $W$ is not reduced, then $W=C'\cup\tilde{L}$, where $C'\subset C_0$ consists of 4 skew lines, and $\tilde{L}$ is a double structure on the remaining line $C_0\setminus C'$, which leads to $p_a(W)=p_a(\tilde{L})-4\le-4$. We end up with contradictions in both cases.

If $\deg(W)=7$, then $p_a(W)=5$ (cf. \cite[Theorem
9.5]{CCJ}), and one must again consider two possibilities. 
Either $W=C_0\cup Q$ where $Q$ is a degree 2 scheme (possibly non-reduced, so $p_a(Q)\le0$), or $W=C'\cup \tilde{L}$, where $C'\subset C_0$ consists of 4 skew lines, and $\tilde{L}$ is a triple structure on the remaining line $C_0\setminus C'$ (so $p_a(\tilde{L})\le1$). In both situations, $p_a(W)\le1$, providing a contradiction as in the previous paragraph.
\end{proof}


\subsection{Existence of foliations}\label{sec-examples}

 The goal of this section is to show that for each $n\in\{1,2,3,4,5\}$, there is a foliation by curves $\sF$ of degree 3 on $\p3$ such that $N_\sF^*(3)$ is an instanton bundle of charge $n$. 
This will allow us to complete the proof of Main Theorem \ref{classification-degree3} and fill out all the information in Table \ref{table deg 3}.

Let  $E$ be an instanton bundle of charge $n$.  
Note that twisting the Euler sequence for the cotangent bundle by $E(3)$ we obtain
$$ 0 \to E\otimes\Omega^1_{\p3}(3) \to  E(2)^{\oplus4} \to E(3) \to 0 $$
from which we conclude that
\begin{equation}\label{dim-hom}
\hom(E(-3),\Omega_{\p3}^1) \geq 4\cdot h^0(E(2))-h^0(E(3)) = 40-11c_2(E),  \end{equation}
with the last equality following from the fact that $h^0(E(p)=\chi(E(p))$ for $p\ge2$ since instanton bundles of charge $n$ are $n$-regular. This last fact also allows us to easily compute $h^1(E(p))$ for every $p\in\Z$; we have (we only write the dimensions of the nonzero cohomologies):
\begin{itemize}
\item for $n=1$, $h^1(E(-1))=1$;
\item for $n=2$, $h^1(E(-1))=h^1(E)=2$;
\item for $n=3$, $h^1(E(-1))=3$, $h^1(E)=4$ and $h^1(E(1))=1+h^0(E(1))$;
\end{itemize}
In addition, if $E$ is an instanton bundle of charge 4 with natural cohomology, then $h^1(E(-1))=4$, $h^1(E)=6$ and $h^1(E(1))=4$.

Moreover, the dimension of the space of global sections of  $E(1)$ can be bounded, as the next result shows.

\begin{Prop}
Let $E$ be a stable rank two vector bundle with $c_1(E)=0$ on $\P^3$. Then either $E$ is a null-correlation bundle, or $h^0(E(1))\leq 2$.
\end{Prop}
\begin{proof}
Let us suppose that $E$ is not a null-correlation bundle. This implies that, restricted to the general hyperplane $H \subset \P^3$, the vector bundle $F := E_{|H}$ is stable as well, see \cite[Theorem 3]{Barth}. Consider the short exact sequence
$$
0 \rightarrow E \rightarrow E(1) \rightarrow F(1) \rightarrow 0.
$$
Being $F$ stable, we have that  $H^0(F)=0$ and, moreover, the following short exact sequence
$$
0 \rightarrow F \rightarrow F(1) \rightarrow \mathcal{O}_L(1)^{\oplus2}\rightarrow 0,
$$
where the surjective map is given by the splitting type of $F$ on a generic line $L$ of $H \simeq \P^2$. Indeed, being $F$ stable, we have that $F_{|L} \simeq \mathcal{O}_L^2$ by the Grauert--M\"ulich theorem (see \cite[Corollary 2 of Theorem II.2.1.4]{OSS}).
Taking the cohomology of the latter short exact sequence, we get that $h^0(F(1))\leq 4$.\\
Recall that $F$ must have at least one jumping line. Indeed, if not, $F$ would be uniform and therefore homogeneous. In this case, $F$ (see \cite{VdeVen}) would either be a direct sum of line bundles, impossible being $F$ stable, or a twist of the tangent bundle, impossible because $c_1(F)=0$.

Hence, having at least one jumping line $\ell$ for $F$, we can consider the following elementary transformation
$$
0 \rightarrow G_\ell \rightarrow F \rightarrow \mathcal{O}_{\ell}(-\alpha) \rightarrow 0,
$$
with $\alpha>0$.\\

In particular, $c_1(G_\ell)=-1$ and $h^0(G_\ell)=0$, which means that $G_\ell$ is stable as well.\\
From here, two possibilities arise:
\begin{description}
\item[Case 1] the bundle $G_\ell$ admits a jumping line. i.e. a line $\ell_1$ such that $(G_\ell)_{|\ell_1} \simeq \mathcal{O}_{\ell_1}(-\beta) \oplus \mathcal{O}_{\ell_1}(\beta-1)$, with $\beta \geq 2$.\\
Applying once again an elementary transformation, we get
$$
0 \rightarrow K_1(-1) \rightarrow G_\ell \rightarrow \mathcal{O}_{\ell_1}(-\beta) \rightarrow 0.
$$
If $h^0(K_1) = 0$, then $h^0(G_\ell(1)) = 0$ as well and $h^0(E(1))\leq h^0(F(1))\leq 1$.\\
Analogously, if $h^0(K_1) = 1$, then $h^0(G_\ell(1)) = 1$ as well and $h^0(E(1))\leq h^0(F(1))\leq 2$.\\
To conclude, let us show that $h^0(K_1) > 1$ is impossible. Indeed, by \cite[Lemma 2']{Barth}, in this case $K_1 \simeq \mathcal{O}_{\mathbb{P}^2} ^{\oplus 2}$ and hence $h^0(G_\ell(1))=2$.\\
Consider a line $\ell_2$, different from $\ell_1$, that gives the splitting  type $(G_\ell)_{|\ell_2} \simeq \mathcal{O}_{\ell_2}(-\gamma) \oplus \mathcal{O}_{\ell_2}(\gamma-1)$, with $\gamma \geq 1$. Considering the elementary transformation
$$
0 \rightarrow K_2(-1) \rightarrow G_\ell \rightarrow \mathcal{O}_{\ell_2}(-\gamma) \rightarrow 0
$$
and, having $h^0(G_\ell(1))=2$, we get that $K_2 \simeq \mathcal{O}_{\mathbb{P}^2} ^{\oplus 2}$ as well. 
Having $\Hom(\mathcal{O}_{\mathbb{P}^2}(-1),\mathcal{O}_{\ell_1}(-\beta)) = 0$, we get the following commutative diagram
$$
\xymatrix{
0 \ar[r] & \mathcal{O}_{\mathbb{P}^2}(-1)^{\oplus 2} \ar[r] \ar[d]_{=} & G_\ell \ar[r] \ar[d]_{=} & \mathcal{O}_{\ell_2}(-\gamma) \ar[r] \ar@{-->}[d] & 0\\
0 \ar[r] & \mathcal{O}_{\mathbb{P}^2}(-1)^{\oplus 2} \ar[r]& G_\ell \ar[r]  & \mathcal{O}_{\ell_1}(-\beta) \ar[r] & 0
}
$$
The third vertical map is an isomorphism as well, leading to a contradiction.

\item[Case 2] the bundle $G_\ell$, obtained for any choice of a jumping line $\ell$, is uniform. Indeed, if there is a line $\ell$ such that $G_\ell$ is not uniform, we apply again the previous case.\\
In this case $G_\ell$ is homogeneous and, by stability, we have $G_\ell \simeq \Omega_{\mathbb{P}^2}(1)$. Recalling that, being $c_1(F)=0$, we have a divisor of jumping lines, and, by assumption, every elementary transformation constructed considering any of these lines leads to a twist of the cotangent bundle.\\

If $\alpha=2$, we consider as before a jumping line $\tilde\ell$, different from $\ell$, such that $G_{\tilde{\ell}} \simeq \Omega_{\mathbb{P}^2}(1)$ as well. 
Having $\Hom(\Omega_{\mathbb{P}^2}(1),\mathcal{O}_{\ell}(-\alpha)) = 0$, we have again a commutative diagram
$$
\xymatrix{
0 \ar[r] & \Omega_{\mathbb{P}^2}(1) \ar[r] \ar[d]_{=} & F \ar[r] \ar[d]_{=} & \mathcal{O}_{\tilde\ell}(-\tilde\alpha) \ar[r] \ar@{-->}[d] & 0\\
0 \ar[r] & \Omega_{\mathbb{P}^2}(1) \ar[r]& F \ar[r]  & \mathcal{O}_{\ell}(-\alpha) \ar[r] & 0
}
$$
which leads to a contradiction.\\

If $\alpha=1$, necessarily  $F$ is a Steiner bundle defined by a short exact sequence
$$
0 \rightarrow \mathcal{O}_{\mathbb{P}^2}(-2)^{\oplus 2} \rightarrow \mathcal{O}_{\mathbb{P}^2}(-1)^{\oplus 4} \rightarrow F \rightarrow 0.
$$
 Indeed, in this case, we have that $F$ fits in the short exact sequence 
$$
0 \rightarrow \Omega_{\mathbb{P}^2}(1) \rightarrow F \rightarrow \mathcal{O}_{\ell}(-1) \rightarrow 0
$$
and, applying the Horseshoe Lemma (see \cite[Lemma 2.2.8]{weibel}), it is possible to find such resolution of $F$ given the ones of $\Omega_{\mathbb{P}^2}(1)$ and $\mathcal{O}_{\ell}(-1)$.

Consider the following exact sequence in cohomology
\begin{equation}\label{eq-restriction}
0 \rightarrow H^0(E(1)) \rightarrow H^0(F(1)) \rightarrow H^1(E) \rightarrow H^1(E(1)) \rightarrow 0 .
\end{equation}
Using Riemann--Roch Theorem and denoting $c=c_2(E)$, we compute $\chi(E)=2-2c$ and $\chi(E(1))=8-3c$. Diagram (\ref{eq-restriction}) implies that $c=2$. From \cite[Lemma 9.3]{H1}, we have that $h^0(E(1))=2$.
\end{description}
Having found that $h^0(E(1))\leq 2$ in all the possible cases, the result is proved.
\end{proof}

\begin{Prop}\label{c_2=1}
For each instanton bundle $E$ of charge 1 there is a foliation by curves $\sF$ of degree 3 such that $N_\sF^*(3)=E$. Furthermore, $\sing(\sF)$ is a curve of degree 8 and arithmetic genus 5 that is connected whenever it is reduced, and $\dim_\C M_\sF=1$.
\end{Prop}

The second statement in the previous proposition follows from the formulas of Theorem \ref{P:SLocus} and in display \eqref{h0(OZ)}, as well as the considerations in the paragraph just above it. 

\begin{proof}
Instanton bundles of charge 1 are 1-regular, so $E(1)$ is globally generated. It follows that $E\otimes\Omega^1_{\p3}(3)$ is globally generated as well; Ottaviani's Bertini type Theorem \cite[Teorema 2.8]{O} implies that there is a monomorphism $E(-3)\to \Omega^1_{\p3}$ whose cokernel is a torsion free sheaf (see also the Appendix in \cite{CCJ}).
\end{proof}

In the remaining part of this section, we will produce  explicit examples  that are obtained by applying the recent techniques introduced by Muniz in \cite{Muniz}, using Macaulay2
 \cite{M2-hyper}. These examples will ensure the existence of the required foliations and therefore close the proof of Main Theorem \ref{classification-degree3}.

\begin{Ex}\label{examample-d7-g2}\rm 
Let us start by considering the following seven lines: 
$$
L_1 = V(x_3,x_1), \: \: L_2 = V(x_3,x_1-x_2), \: \:L_3 = V(x_2,x_0), \: \:L_4 = V(x_2,x_0-x_3)
$$
$$
L_5 = V(x_1,x_0), \:\: L_6 = V(x_1-x_3,x_0-x_2), \:\: L_7 = V(x_1+x_3,x_0+x_2).
$$ 
Then $C = L_1 \cup \dots \cup L_7$ is a degree $7$ curve and one may check that its genus is two. Now let $v = a_0\frac{\partial }{\partial x_0} + \dots + a_3\frac{\partial }{\partial x_3}$ be the vector field given by 

\begin{align*}
   a_0 &= -3\,x_{0}^{2}x_{1}-3\,x_{0}x_{1}x_{2}-6\,x_{1}^{2}x_{2}+x_{0}x_{2}^{2}+6\,x_{1}x_{2}^{2}\\ & \qquad+3\,x_{0}^{2}x_{3}+3\,x_{0}x_{1}x_{3}+3\,x_{0}x_{2}x_{3}+6\,x_{1}x_{2}x_{3}-5\,x_{0}x_{3}^{2} \\
   a_1 & = 3\,x_{0}x_{1}^{2}-6\,x_{0}x_{1}x_{2}+3\,x_{1}^{2}x_{2}-5\,x_{1}x_{2}^{2}\\ & \qquad+6\,x_{0}^{2}x_{3}-3\,x_{0}x_{1}x_{3}-3\,x_{1}^{2}x_{3}+6\,x_{0}x_{2}x_{3}+9\,x_{1}x_{2}x_{3}-6\,x_{0}x_{3}^{2}+x_{1}x_{3}^{2} \\
   a_2 &= -3\,x_{0}x_{1}x_{2}-3\,x_{1}x_{2}^{2}+x_{2}^{3}+12\,x_{0}^{2}x_{3}+3\,x_{0}x_{2}x_{3}\\ & \qquad+9\,x_{1}x_{2}x_{3}-3\,x_{2}^{2}x_{3}-12\,x_{0}x_{3}^{2}+x_{2}x_{3}^{2} \\
   a_3 &= 6\,x_{1}^{2}x_{2}-6\,x_{1}x_{2}^{2}+3\,x_{0}x_{1}x_{3}+6\,x_{0}x_{2}x_{3}+3\,x_{1}x_{2}x_{3} \\ & \qquad +x_{2}^{2}x_{3}-3\,x_{0}x_{3}^{2}-3\,x_{1}x_{3}^{2}-3\,x_{2}x_{3}^{2}+x_{3}^{3} \\
\end{align*}
We claim that $v$ is singular precisely along $C$, and one may check this with Macaulay2 \cite{M2-hyper}. In the script below we define $C$ as above and $Z$ as the singular scheme of $v$ and then compare them.

\begin{center} 
\begin{BVerbatim}
R = QQ[x_0..x_3];
C = intersect(ideal(x_3,x_1), ideal(x_3,x_1-x_2), ideal(x_2,x_0), 
ideal(x_2,x_0-x_3), ideal(x_1,x_0), ideal(x_1-x_3,x_0-x_2), 
ideal(x_1+x_3,x_0+x_2));
a0 = -3*x_0^2*x_1-3*x_0*x_1*x_2-6*x_1^2*x_2+x_0*x_2^2+6*x_1*x_2^2
+3*x_0^2*x_3+3*x_0*x_1*x_3 +3*x_0*x_2*x_3+6*x_1*x_2*x_3-5*x_0*x_3^2;
a1 = 3*x_0*x_1^2-6*x_0*x_1*x_2+3*x_1^2*x_2-5*x_1*x_2^2+6*x_0^2*x_3
-3*x_0*x_1*x_3-3*x_1^2*x_3+6*x_0*x_2*x_3+9*x_1*x_2*x_3-6*x_0*x_3^2
+x_1*x_3^2;
a2 = -3*x_0*x_1*x_2-3*x_1*x_2^2+x_2^3+12*x_0^2*x_3+3*x_0*x_2*x_3
+9*x_1*x_2*x_3-3*x_2^2*x_3-12*x_0*x_3^2+x_2*x_3^2;
a3 = 6*x_1^2*x_2-6*x_1*x_2^2+3*x_0*x_1*x_3+6*x_0*x_2*x_3
+3*x_1*x_2*x_3+x_2^2*x_3-3*x_0*x_3^2-3*x_1*x_3^2-3*x_2*x_3^2+x_3^3;
Z = saturate minors(2, matrix{{x_0,x_1,x_2,x_3},{a0,a1,a2,a3}})
C == Z 
\end{BVerbatim}
\end{center}

\end{Ex}



\begin{Ex} \label{examample-2cubics}\rm 
Now consider the curve $C = C_1 \sqcup C_2$ the disjoint union of two twisted cubics given by the maximal minors of the matrices
\[
 \begin{pmatrix}
x_0 & x_1 & x_2 \\ x_1 & x_2 & x_3
\end{pmatrix} 
\quad \text{and} \quad 
\begin{pmatrix}
x_2 & x_3 & x_0 \\ x_3 & x_0 & -x_1
\end{pmatrix}.
\]
For this choice of $C$ we can produce the vector field $v = a_0\frac{\partial }{\partial x_0} + \dots + a_3\frac{\partial }{\partial x_3}$ such that
\begin{align*}
    a_0 & = -x_{0}^{3}-9\,x_{0}^{2}x_{1}+x_{0}x_{1}^{2}-6\,x_{0}^{2}x_{2}+6\,x_{1}^{2}x_{2}-x_{0}x_{2}^{2}+12\,x_{1}x_{2}^{2}-18\,x_{1}^{2}x_{3}\\ & \qquad+9\,x_{0}x_{2}x_{3}-6\,x_{1}x_{2}x_{3}+5\,x_{0}x_{3}^{2} \\
    a_1 & = -x_{0}^{2}x_{1}+9\,x_{0}x_{1}^{2}+x_{1}^{3}-18\,x_{0}^{2}x_{2}+5\,x_{1}x_{2}^{2}-6\,x_{0}^{2}x_{3}-6\,x_{1}^{2}x_{3}+6\,x_{0}x_{2}x_{3}\\ & \qquad-9\,x_{1}x_{2}x_{3}+12\,x_{0}x_{3}^{2}-x_{1}x_{3}^{2} \\
    a_2 & = -x_{0}^{2}x_{2}-3\,x_{0}x_{1}x_{2}-5\,x_{1}^{2}x_{2}+6\,x_{0}x_{2}^{2}-x_{2}^{3}-6\,x_{0}^{2}x_{3}-6\,x_{0}x_{1}x_{3}-3\,x_{2}^{2}x_{3}\\ & \qquad+6\,x_{0}x_{3}^{2}+6\,x_{1}x_{3}^{2}-x_{2}x_{3}^{2} \\
    a_3 & =  -6\,x_{0}x_{1}x_{2}-6\,x_{1}^{2}x_{2}-6\,x_{0}x_{2}^{2}-6\,x_{1}x_{2}^{2}+5\,x_{0}^{2}x_{3}+3\,x_{0}x_{1}x_{3}+x_{1}^{2}x_{3}-x_{2}^{2}x_{3}\\ & \qquad+6\,x_{1}x_{3}^{2}+3\,x_{2}x_{3}^{2}-x_{3}^{3}\\
\end{align*}

To see that $C$ is the singular scheme of $v$ one may use the following script.

\begin{center} 
\begin{BVerbatim}
R = QQ[x_0..x_3];
C1 = minors(2, matrix{{x_0,x_1,x_2},{x_1,x_2,x_3}});
C2 = minors(2, matrix{{x_2,x_3,x_0},{x_3,x_0,-x_1}});
saturate(C1+C2)  -- check if they are disjoint
C = intersect(C1,C2);
a0 = -x_0^3-9*x_0^2*x_1+x_0*x_1^2-6*x_0^2*x_2+6*x_1^2*x_2-x_0*x_2^2
+12*x_1*x_2^2-18*x_1^2*x_3+9*x_0*x_2*x_3-6*x_1*x_2*x_3+5*x_0*x_3^2;
a1 = -x_0^2*x_1+9*x_0*x_1^2+x_1^3-18*x_0^2*x_2+5*x_1*x_2^2-6*x_0^2*x_3
-6*x_1^2*x_3+6*x_0*x_2*x_3-9*x_1*x_2*x_3+12*x_0*x_3^2-x_1*x_3^2;
a2 = -x_0^2*x_2-3*x_0*x_1*x_2-5*x_1^2*x_2+6*x_0*x_2^2-x_2^3-6*x_0^2*x_3
-6*x_0*x_1*x_3-3*x_2^2*x_3+6*x_0*x_3^2+6*x_1*x_3^2-x_2*x_3^2;
a3 = -6*x_0*x_1*x_2-6*x_1^2*x_2-6*x_0*x_2^2-6*x_1*x_2^2+5*x_0^2*x_3
+3*x_0*x_1*x_3+x_1^2*x_3-x_2^2*x_3+6*x_1*x_3^2+3*x_2*x_3^2-x_3^3;
Z = saturate minors(2, matrix{{x_0,x_1,x_2,x_3},{a0,a1,a2,a3}});
C == Z
\end{BVerbatim}
\end{center}
 \end{Ex} 
 

\begin{Ex}\label{examample-quartic-lines}\rm 
 Consider the quartic elliptic curve   $$C_1=V(x_0x_3-x_1x_2, x_2x_3-x_0x_1)=V(x_{3},\,x_{1})\cup V\left(x_{2},\,x_{0}\right)\cup V\left(x_{1}-x_{3},\,x_{0}-x_{2}\right) \cup  V\left(x_{1}+x_{3},\,x_{0}+x_{2}\right)  $$
and the skew lines 
$$
L_1=V\left(x_{2}+x_{3},\,x_{0}-x_{1}+x_{3}\right), \: L_2=V\left(x_{1}+x_{2}+x_{3},\,x_{0}-x_{3}\right). 
 $$
The curve $C=C_1\sqcup L_1 \sqcup L_2$ has $3$ connected components, i.e., $h^0(\mathcal{O}_C)=3$, degree $6$ and arithmetic genus $-1$.
 We can produce the vector field $v = a_0\frac{\partial }{\partial x_0} + \dots + a_3\frac{\partial }{\partial x_3}$ such that
\\
\\
\begin{align*}
a_0= &  -6\,x_{0}^{3}-13\,x_{0}^{2}x_{1}+49\,x_{0}x_{1}^{2}-13\,x_{0}^{2}x_{2}+38\,x_{0}x_{1}x_{2}+42\,x_{1}^{2}x_{2}-11\,x_{0}x_{2}^{2}+42\,x_{1}x_{2}^{2}\\ 
& \qquad-5\,x_{0}^{2}x_{3}-11\,x_{0}x_{1}x_{3}-59\,x_{0}x_{2}x_{3}-24\,x_{1}x_{2}x_{3}-48\,x_{2}^{2}x_{3}-32\,x_{0}x_{3}^{2}-24\,x_{2}x_{3}^{2},
\\
a_1= & 
6\,x_{0}^{2}x_{1}+35\,x_{0}x_{1}^{2}-11\,x_{1}^{3}+47\,x_{0}x_{1}x_{2}+20\,x_{1}^{2}x_{2}+49\,x_{1}x_{2}^{2}+24\,x_{0}^{2}x_{3}+25\,x_{0}x_{1}x_{3}\\
& \qquad +7\,x_{1}^{2}x_{3}-72\,x_{0}x_{2}x_{3}+7\,x_{1}x_{2}x_{3}-84\,x_{2}^{2}x_{3}-72\,x_{0}x_{3}^{2}+4\,x_{1}x_{3}^{2}-60\,x_{2}x_{3}^{2},
\\
a_2= & 12\,x_{0}^{2}x_{1}-12\,x_{0}x_{1}^{2}-6\,x_{0}^{2}x_{2}-19\,x_{0}x_{1}x_{2}-5\,x_{1}^{2}x_{2}-13\,x_{0}x_{2}^{2}-34\,x_{1}x_{2}^{2}-11\,x_{2}^{3}-30\,x_{0}^{2}x_{3}\\
& \qquad -18\,x_{0}x_{1}x_{3}+7\,x_{0}x_{2}x_{3}-17\,x_{1}x_{2}x_{3}+19\,x_{2}^{2}x_{3}+36\,x_{0}x_{3}^{2}+16\,x_{2}x_{3}^{2},
\\
a_3= & -6\,x_{0}^{2}x_{1}+6\,x_{0}x_{1}^{2}-30\,x_{0}x_{1}x_{2}-12\,x_{1}^{2}x_{2}-12\,x_{1}x_{2}^{2}+18\,x_{0}^{2}x_{3}-25\,x_{0}x_{1}x_{3}-11\,x_{1}^{2}x_{3}\\
& \qquad +5\,x_{0}x_{2}x_{3}-10\,x_{1}x_{2}x_{3}-5\,x_{2}^{2}x_{3}-11\,x_{0}x_{3}^{2}+7\,x_{1}x_{3}^{2}+7\,x_{2}x_{3}^{2}+4\,x_{3}^{3}
\end{align*}
 To see that $C$ is the singular scheme of $v$ one may use the following script.

\begin{center} 
\begin{BVerbatim}
R = QQ[x_0..x_3];
C1 = ideal(-x_1x_2+x_0x_3,-x_0x_1+x_2x_3);
L1 = ideal(x_2+x_3,x_0-x_1+x_3);
L2 = ideal(x_1+x_2+x_3,x_0-x_3);
saturate(C1+L1+L2)  -- check if they are disjoint
C = intersect(C1,L1,L2);
a0 = -6*x_0^3-13*x_0^2*x_1+49*x_0*x_1^2-13*x_0^2*x_2+38*x_0*x_1*x_2+
42*x_1^2*x_2-11*x_0*x_2^2+42*x_1*x_2^2-5*x_0^2*x_3-11*x_0*x_1*x_3-
59*x_0*x_2*x_3-24*x_1*x_2*x_3-48*x_2^2*x_3-32*x_0*x_3^2-24*x_2*x_3^2;
a1 = 6*x_0^2*x_1+35*x_0*x_1^2-11*x_1^3+47*x_0*x_1*x_2+20*x_1^2*x_2+
49*x_1*x_2^2+24*x_0^2*x_3+25*x_0*x_1*x_3+7*x_1^2*x_3-72*x_0*x_2*x_3+
7*x_1*x_2*x_3-84*x_2^2*x_3-72*x_0*x_3^2+4*x_1*x_3^2-60*x_2*x_3^2; 
a2 = 12*x_0^2*x_1-12*x_0*x_1^2-6*x_0^2*x_2-19*x_0*x_1*x_2-5*x_1^2*x_2-
13*x_0*x_2^2-34*x_1*x_2^2-11*x_2^3-30*x_0^2*x_3-18*x_0*x_1*x_3+
7*x_0*x_2*x_3-17*x_1*x_2*x_3+19*x_2^2*x_3+36*x_0*x_3^2+16*x_2*x_3^2; 
a3 = -6*x_0^2*x_1+6*x_0*x_1^2-30*x_0*x_1*x_2-12*x_1^2*x_2-12*x_1*x_2^2+
18*x_0^2*x_3-25*x_0*x_1*x_3-11*x_1^2*x_3+5*x_0*x_2*x_3-10*x_1*x_2*x_3-
5*x_2^2*x_3-11*x_0*x_3^2+7*x_1*x_3^2+7*x_2*x_3^2+4*x_3^3;  
Z = saturate minors(2, matrix{{x_0,x_1,x_2,x_3},{a0,a1,a2,a3}});
C == Z
\end{BVerbatim}
\end{center}

\end{Ex}



\begin{Ex} \label{examample-2double-lines}\rm 
Let $C = C_1 \sqcup C_2$ where $C_1$ and $C_2$ are double lines of genus $-3$ given by
\begin{align*}
    I_{C_1} & = \left(x_{0}^{2},\,x_{0}x_{1},\,x_{1}^{2},\,x_{0}x_{2}^{3}-x_{1}x_{3}^{3}\right),\\
    I_{C_2} & = \left(x_{2}^{2},\,x_{2}x_{3},\,x_{3}^{2},\,x_{0}^{3}x_{2}-x_{1}^{3}x_{3}\right).
\end{align*}
Then we can define the vetor field $v = a_0\frac{\partial }{\partial x_0} + \dots + a_3\frac{\partial }{\partial x_3}$ such that 
\begin{align*}
    a_0 & = -x_{0}^{2}x_{2}-x_{0}x_{1}x_{2}-x_{0}^{2}x_{3}+x_{0}x_{1}x_{3}-2\,x_{1}^{2}x_{3}  \\
    a_1 & = -2\,x_{0}^{2}x_{2}-x_{0}x_{1}x_{2}-x_{1}^{2}x_{2}-x_{0}x_{1}x_{3}+x_{1}^{2}x_{3}  \\
    a_2 & =x_{0}x_{2}^{2}+x_{1}x_{2}^{2}+x_{0}x_{2}x_{3}-x_{1}x_{2}x_{3}-2\,x_{1}x_{3}^{2} \\
    a_3 & =  -2\,x_{0}x_{2}^{2}+x_{0}x_{2}x_{3}+x_{1}x_{2}x_{3}+x_{0}x_{3}^{2}-x_{1}x_{3}^{2}\\
\end{align*}

To see that $C$ is the singular scheme of $v$ one may use the following script.

\begin{center} 
\begin{BVerbatim}
R = QQ[x_0..x_3];
C1 = ideal(x_0^2, x_0*x_1, x_1^2, x_0*x_2^3 - x_1*x_3^3);
C2 = ideal(x_2^2, x_2*x_3, x_3^2, x_2*x_0^3 - x_3*x_1^3); 
saturate(C1+C2) --check that they are disjoint
C = intersect(C1,C2);
a0 = -x_0^2*x_2-x_0*x_1*x_2-x_0^2*x_3+x_0*x_1*x_3
-2*x_1^2*x_3;
a1 = -2*x_0^2*x_2-x_0*x_1*x_2-x_1^2*x_2-x_0*x_1*x_3
+x_1^2*x_3;
a2 = x_0*x_2^2+x_1*x_2^2+x_0*x_2*x_3-x_1*x_2*x_3
-2*x_1*x_3^2;
a3 = -2*x_0*x_2^2+x_0*x_2*x_3+x_1*x_2*x_3+x_0*x_3^2
-x_1*x_3^2;
Z = minors(2, matrix{{x_0,x_1,x_2,x_3},{a0,a1,a2,a3}});
C == saturate Z
\end{BVerbatim}
\end{center}
\end{Ex} 

The previous examples give us the following result.

\begin{Prop}
${}$
\begin{enumerate}
    \item \label{c_2=2}
 There exists a foliation by curves $\sF$ of degree 3 such that $N_\sF^*(3)$ is an instanton bundle of charge 2.  Furthermore, $\sing(\sF)$  consists of the union  of   7 lines with  arithmetic genus  2, and $\dim_\C M_\sF=4$. 
 \item \label{c_2=3}There exists a foliation by curves $\sF$ of degree 3 such that $N_\sF^*(3)=E$ is an instanton bundle of charge 3 such that $h^0(E(1))=0$.  Furthermore, $\sing(\sF)$  consists of the disjoint union of two twisted cubics and $\dim_\C M_\sF=8$.
 \item There exists a foliation by curves $\sF$ of degree 3 such that $N_\sF^*(3)=E$ is an instanton bundle of charge 3 such that $h^0(E(1))=1$.  Furthermore, $\sing(\sF)$  consists of the  disjoint  union  of a quartic elliptic curve and two skew  lines, and $\dim_\C M_\sF=9$. 
 \item There exists a foliation by curves $\sF$ of degree 3 such that $N_\sF^*(3)$ is an instanton bundle of charge 4 with natural cohomology. Furthermore, $\sing(\sF)$ consists of  5 disjoint lines and $\dim_\C M_\sF=14$.
 \item \label{c_2=5} There exists a foliation by curves $\sF$ of degree 3 such that $N_\sF^*(3)$ is an instanton bundle of charge 5.  Furthermore, $\sing(\sF)$  consists of the union   2  double  lines with  arithmetic genus  $-3$. 
 \end{enumerate}
\end{Prop}

\begin{proof}
Items (1), (2), (3), and (5) are given by 
examples \ref{examample-d7-g2}, \ref{examample-2cubics},\ref{examample-quartic-lines} and \ref{examample-2double-lines}.  On the other hand, regarding item (4), Lemma \ref{lemma-disjoint} already guarantees the existence of a foliation by curves of degree 3 satisfying the conditions claimed.
\end{proof}

\begin{Obs}\rm 
It follows from our results that a foliation by curves $\sF$ whose singular locus has pure dimension 1 has   its conormal bundle $N_\sF^*$  either split or stable. If $N_\sF^*$ is stable, then by Kobayashi--Hitchin correspondence $N_\sF^*$ admits a K\"ahler--Einstein metric. Thus, it would be interesting to know if such a metric is invariant under holonomy, that is, if the foliation is transversely  K\"ahler--Einstein, see \cite[Chapter 5]{Tondeur}. 
\end{Obs}


\section{Legendrian foliations}\label{sec:legendrian}

A curve  $C\subset \P^3$ is called \emph{Legendrian} if it is tangent to a contact structure given by 
$$\sD: 0\to T_\sD \to T\P^3\to \O(2) \to 0 $$
We recall that $T_\sD=N(1)$, where $N$ is a null-correlation bundle. 

\begin{Def}
A foliation by curves $\sF$ on $\P^3$ is called by Legendrian if $T_{\sF}\subset T_\sD $. That is, the leaves of $\sF$ are Legendrian curves outside the singular set of $\sF$.  
\end{Def}
See \cite{Bry} for details about Legendrian curves on $\P^3$. 

\begin{Theorem}\label{thm-leg}
Every Legendrian foliation $\sF$ by curves of degree $d$ is of the form $\omega_0\wedge\omega$, where $\omega_0$ is a  contact form and $\omega\in H^0(\Omega^1_{\p3}(d+1))$. In addition, the moduli space of the Legendrian foliations of degree $d$ is an irreducible quasi-projective variety of dimension
$$d\cdot\binom{d+3}{2} - \binom{d+2}{3}+4\:\:\mbox{ if } \:\: d \geq 2$$
and of dimension 8 if $d=1$.
\end{Theorem}

In particular, it follows from \cite[Theorem 2]{CJV} that the singular scheme of a Legendrian foliation is a Buchsbaum curve which, by Proposition \ref{con}, is connected for $d\ge2$.

\begin{proof}
In fact,  we have an induced section $\sigma: T_{\sF}=\O(1-d) \to  T_\sD=N(1)$ where $N$ is a null-correlation bundle. That is, the Legendrian foliation $\sF$ induces a global section of $\sigma\in H^0(\P^3, N(d))$, which fits into the following commutative diagram
$$
\xymatrix{
& 0 \ar[d] & 0\ar[d]\\
 & \O(1-d) \ar[d]^{\sigma} \ar@{=}[r] & \O(1-d) \ar[d] \\
0 \ar[r]& N(1) \ar[r] \ar[d] & T\P^3 \ar[r] \ar[d] & \O(2) \ar[r] \ar@{=}[d] & 0 \\
0 \ar[r] & \IC(d+1) \ar[d] \ar[r] & G \ar[d] \ar[r] & \O(2) \ar[r] & 0\\
& 0 & 0
}
$$
where  the curve $C$ is the zero locus of $\sigma$, and the central column is the one defining the Legendrian foliation. Considering the dual exact sequence of the bottom row, we have
$$
0 \rightarrow \O(-2) \rightarrow G^* \rightarrow \O(-1-d) \rightarrow 0
$$
which directly implies that $G^* \simeq \O(-2) \oplus \O(-1-d)$. Therefore we can conclude that the Legendrian foliation is of the required form  $\omega_0\wedge\omega$, where $\omega_0$ is the contact form such that $\ker(\omega_0)=T_\sD=N(1)$ and $\omega\in H^0(\Omega^1_{\p3}(d+1))$.

In order to prove the last statements of the result, we want to apply Theorem \ref{quot-thm}. First of all, we need the vanishings required in the mentioned result, which means we must compute
$$
\dim \left(\ext^1(G^*,\mathcal{I}_Z(d-1))\right) = h^1(\mathcal{I}_Z(d+1)) + h^1(\mathcal{I}_Z(2d))  =0,
$$
where $Z$ denotes the singular locus of the Legendrian foliation. Hence, we can apply the theorem and prove the statement, recalling that the dimension of the moduli space is given by
\begin{equation}\label{modLeg}
\dim \left(\Hom(G^*,\mathcal{I}_Z(d-1))\right) =h^0(\mathcal{I}_Z(d+1)) + h^0(\mathcal{I}_Z(2d)) ,
\end{equation}
which is equal to
$$
d\cdot\binom{d+3}{2} - \binom{d+2}{3}+4\:\:\mbox{ if } \:\: d \geq 2$$
and to 8 if $d=1$, as required. 

\end{proof}

\begin{Obs}\rm
As proved in the previous result, every Legendrian foliation can be expressed  as the wedge product of a contact form with a 1-form $\omega\in  H^0(\Omega^1_{\p3}(d+1)))$. This means that the space Legendrian foliations of degree $d$ is an open set in the image of the linear map (given by the wedge product)
$$
 H^0(\Omega^1_{\p3}(2)) \otimes H^0(\Omega^1_{\p3}(d+1)) \to H^0(\Omega^2_{\p3}(d+3)).
 $$
 We have a natural surjective map 
$$ \begin{array}{c}
       \left\{ \mbox{Legendrian foliations}\right\}  \subset \mathbb{P}H^0(\Omega^2_{\p3}(d+3)))\\
       \downarrow \ \ \ \ \ \ \\
      \left\{ \mbox{Contact forms}\right \} \subset \mathbb{P}H^0(\Omega^1_{\p3}(2))) \simeq  \mathbb{P}^4  
\end{array} $$
defined by $\omega_0\wedge \omega \to \omega_0$.
The fiber over a contact  structure $w_0$ is an open set of  $$
\frac{H^0(\Omega^1_{\p3}(d+1)))}{\{h\omega_0;\ f\in H^0(\O(d-1)) \}}.$$
 Observe that 
$$
\dim \frac{H^0(\Omega^1_{\p3}(d+1)))}{\{h\omega_0;\ f\in H^0(\O(d-1) )\}}=
4\binom{d+3}{3}-\binom{d+4}{3}-\binom{d+2}{3}=d\cdot\binom{d+3}{2}
-\binom{d+2}{3},
$$
since $$4\binom{d+3}{3}-\binom{d+4}{3}=d\cdot\binom{d+3}{2}.$$ 
Such a description fits perfectly with the computation of the dimension of the moduli space.
\end{Obs}

\begin{Obs}\rm
For the degree 1 case we have the following geometric description: consider the vector space $V:=H^0(\Omega^1_{\p3}(2))$. Since a degree 1 Legendrian foliation is induced by a polynomial 2-form $\omega_0\wedge w$, where $\omega_0,  w \in V$ are generic 1-forms of degree 1, then the space of Legendrian foliations of degree 1 is a Zariski open subset of the Grassmannian $Gr(2, V)$, whose dimension is $8$.      
\end{Obs}

\begin{Obs}\rm
In \cite{CV} the first named author and I. Vainsencher  gave a different description of the space of Legendrian foliations. The authors provide   formulas for the  degrees of the varieties of Legendrian foliations, and also  of the varieties of foliations tangent to a pencil of planes.
\end{Obs}
 
\medskip

\begin{Obs}\label{Obs-moduli-split}\rm 
Another consequence of the previous result is that the conormal sheaf of a Legendrian foliation of degree $d$ always splits as a sum of line bundles $\op3(-2)\oplus\op3(-1-d)$; in other words, Legendrian foliation are of global complete intersection type.

More generally, one can also consider foliations by curves of global complete intersection type given by exact sequences of the form
\begin{equation}\label{gl inter}
0 \longrightarrow \op3(-2-d_1)\oplus\op3(-2-d_2) \stackrel{\phi}{\longrightarrow} \Omega^1_{\p3} \longrightarrow \IC(d_1+d_2) \longrightarrow 0,
\end{equation}
with $d_1,d_2\ge0$. Such foliations have degree $d_1+d_2+1$ and their singular schemes are, by Proposition \ref{con}, connected curves satisfying
$$ \deg(C) = (d_1+d_2)^2-d_1d_2+2(d_1+d_2+1) \:\:\mbox{ and } $$
$$ p_a(C) = (d_1+d_2+1)^3-2(d_1+d_2+1)^2-(d_1+d_2)(3d_1d_2-2)/2,$$
according to the formulas in displays \eqref{eq:degZ} and \eqref{eq:gen}.
In addition, since $\Omega^1(d_1+2)\oplus\Omega^1(d_2+2)$ is globally generated, Ottaviani's Bertini type Theorem \cite[Teorema 2.8]{O} implies that the singular scheme of a generic foliation of the form described in display \eqref{gl inter} is smooth. The moduli spaces of foliations by curves of the form \eqref{gl inter} can be described similarly to Legendrian foliations, following the arguments in the proof of Theorem \ref{thm-leg} above.
\end{Obs}


\section{The Rao module of singular schemes of foliations by curves}\label{sec:buchsbaum}

As mentioned in the Introduction, the first cohomology module of the conormal sheaf (or equivalently, the Rao module of the singular scheme) is an important piece of algebraic information attached to a foliation by curves. We will now focus on explaining the relation between them, and discuss how they can be useful in the classification of particular classes of foliations by curves.

So let $\sF$ be a foliation by curves on $\p3$, and consider the following two graded modules
$$ R_{\sF}:=H^1_*(\IZ) ~~{\rm and} ~~
M_{\sF}:=H^1_*(N_\sF^*); $$
they will be called the \emph{Rao module} and the \emph{first cohomology module} of the foliation $\sF$. Note that $R_{\sF}$ is finite dimensional (as a $\C$-vector space) if and only if $Z=C$ has pure dimension 1, or equivalently, if and only if $\sF$ is of local complete intersection type. On the other hand, $M_{\sF}$ is always finite dimensional.

\begin{Lemma}\label{l:rao}
If $\sF$ is a foliation by curves on $\p3$ of local complete intersection type, then
\begin{equation}\label{ineq-rao}
\dim_{\C} M_{\sF} \le \dim_{\C} R_{\sF} \le \dim_{\C} M_{\sF} + 1.
\end{equation}
Moreover, if $h^1(N_\sF^*)=0$, then the second inequality is an equality.
\end{Lemma}
\begin{proof}
Starting with the exact sequence
$$ 0 \to N_\sF^*(k-d+1) \to \Omega^1(k-d+1) \to \IC(k) \to 0 $$
where $d$ is the degree of $\sF$, it is easy to see that 
$$ h^1(\IC(k)) = h^2
(N_\sF^*(k-d+1)) = h^1(N_\sF^*(2d+2-k)). $$
whenever $k\ne d-1$. When $k=d-1$, we obtain the following exact sequence in cohomology:
$$ 0\to H^0(\IC(d-1)) \to H^1(N_\sF^*) \to H^1(\Omega^1_{\p3}) \to H^1(\IC(d-1)) \to H^2(N_\sF^*) \to 0, $$
thus $h^1(\IC(d-1))-h^1(N_\sF^*(d+3))$ is either 0 or 1, proving the two inequalities in the claim. If $h^1(N_\sF^*)=0$, then $h^1(\IC(d-1))-h^1(N_\sF^*(d+3))=1$ and we get that the second inequality is an equality.
\end{proof}

It is easy to see that the Rao module $R_\sF$ of a foliation of local complete intersection type is always nontrivial. Indeed, if $\dim_{\C} R_{\sF}=0$, then $\dim_{\C} M_{\sF}=0$ as well, implying, by the Horrocks' splitting criterion, that $N_\sF^*$ splits as a sum of line bundles; but then $h^1(N_\sF^*)=0$, hence the second equality in display \eqref{ineq-rao} must be an equality, which yields a contradiction. 

Note also that if $N_\sF^*$ splits as a sum of line bundles, then $\dim_{\C} M_{\sF}=0$ and $h^1(N_\sF^*)=0$, so it follows from Lemma \ref{l:rao} that $\dim_{\C} R_{\sF}=1$. This claim and its converse were already established as a particular case of \cite[Theorem 2]{CJV}; for the sake of completeness, we reproduce the result here.

\begin{Theorem}
Let $\sF$ be a foliation by curves. The conormal sheaf $N_\sF^*$ splits as a sum of line bundles if and only if $\dim_{\C} R_{\sF}=1$.
\end{Theorem}

The goal of this section is to consider the case $\dim_{\C} R_{\sF}=2$, and establish the proof of Main Theorem \ref{mthm4}.

First, recall that a curve $C\subset\p3$ is said to be \emph{arithmetically Buchsbaum} if its Rao module $H^1_*(\IC)$ is trivial as graded $\C[x_0,x_1,x_2,x_3]$-module, that is the multiplication map
$$ H^1(\IC(k))\stackrel{f}{\rightarrow}H^1(\IC(k+1)) $$
is zero for every $f\in H^0(\op3(1))$. A particular class of arithmetically Buchsbaum curves are those with Rao module is concentrated in a single degree, that is, there is $\delta\in\Z$ such that $h^1(\IC(k))=0$ for every $k\ne\delta$. In particular, of $\dim_\C H^1_*(\IC)=1$, then $C$ must be arithmetically Buchsbaum, so the singular scheme of a foliation by curves of global complete intersection type is always arithmetically Buchsbaum, as it was observed in \cite{CJV}.

Another class of foliations by curves with arithmetically Buchsbaum singular schemes arises as follows. Recall that the \emph{null-correlation bundle} $N$ on $\p3$ is defined by the exact sequence
\begin{equation}\label{nc-defn}
0\to \op3(-1) \to \Omega^1_{\p3} \to N \to 0 .
\end{equation}
Note that $N(k)$ is globally generated for every $k\ge1$, so that $N\otimes\Omega^1_{\p3}(k+2)$ is also globally generated for every $k\ge1$. It then follows from Ottaviani's Bertini type Theorem \cite[Teorema 2.8]{O}, that there is a monomorphism $N(-2-k)\to \Omega^1_{\p3}$ for each $k\ge1$ whose cokernel is a torsion free sheaf.

This observation guarantees the existence of foliations by curves
\begin{equation}\label{nc-normal}
0 \to N(-k-2) \to \Omega^1_{\p3} \to \IC(2k) \to 0
\end{equation}
of degree $2k+1$, $k\ge1$, whose conormal sheaves are a twisted null-correlation bundle. The equalities in Theorem \ref{P:SLocus} yield
$$ \deg(C) = (3k+1)(k+1) ~~~~{\rm and} $$
$$ p_a(C) = 5k^3 + 4k^2 -3k - 1. $$
Furthermore, since null-correlation bundles coincide with instantons bundles of charge 1, we remark that the case $k=1$ coincides with the foliations obtained in item (3) of Main Theorem \ref{classification-degree3} for $c_2(E)=1$; see also Proposition \ref{c_2=1}. 

On the other hand, note that there are no morphisms $N(-2-k)\to \Omega^1_{\p3}$ when $k\le-1$; just twist the Euler sequence
$$ 0\to\Omega^1\to\op3(-1)^{\oplus4}\to\op3\to0 $$
by $N(k+2)$ and recall that $H^0(N(k+1))=0$ when $k\le-1$.

For $k=0$, even though $\dim\Hom(N(-2),\Omega^1_{\p3})\ne0$, the curve $C$ in display \eqref{nc-normal} would have degree 1 and genus $-1$, which is impossible; so there can be no monomorphisms $N(-2)\to\Omega^1_{\p3}$ with torsion free cokernel.

\begin{Prop}\label{nc=>dim2}
Let $\sF$ be a foliation by curves on $\p3$. If $N_\sF^*$ is a twisted null-correlation bundle, then $\sing(\sF)$ is a connected, arithmetically Buchsbaum curve and $\dim_\C R_\sF=2$. In addition, for a generic foliation of this type, $\sing(\sF)$ is smooth.
\end{Prop}

\begin{proof}
Consider a foliation by curves like the one in display \eqref{nc-normal}, so that $N_\sF^*=N(-k-2)$ for some null-correlation bundle $N$ and some $k\ge1$. It follows that $\dim_\C M_\sF=1$, and Lemma \ref{l:rao} implies that $\dim_\C R_\sF=2$ because $h^0(N_\sF^*)=h^0(N(-k-2))=0$ for every $k\ge 1$. The connectedness of $\sing(\sF)$ is a simple consequence of Proposition \ref{con} since $h^2(N_\sF^*(1-d))=h^2(N(-3k-2))=0$ for every $k\ge1$. The smoothness of $\sing(\sF)$ for a generic foliation is an immediate consequence of Ottaviani's Bertini type Theorem \cite[Teorema 2.8]{O}.

In addition, we show that $C:=\sing(\sF)$ is arithmetically Buchsbaum. First, we check that $H^1(\IC(p))=0$ for $p\ne2k,3k-1$. Indeed, the cohomology sequence associated to the exact sequence in display \eqref{nc-normal} yields
$$ H^1(\Omega^1_{\p3}(p-2k)) \to H^1(\IC(p)) \to H^2(N(p-3k-2)). $$
The left term vanishes when $p\ne 2k$, while the right one vanishes when $p\ne 3k-1$. 

If $k=1$ (so that $2k=3k-1=2$), then we have
$$ 0 \to H^1(\Omega^1_{\p3}) \to H^1(\IC(2)) \to H^1(N(-3)) \to 0, $$
thus $h^1(\IC(p))=0$ for $p\ne 2$, and $h^1(\IC(2))=2$; in particular, $C$ must be arithmetically Buchsbaum.

If $k\ge2$, then $h^1(\IC(p))=0$ for $p\ne2k,3k-1$, and $h^1(\IC(2k))=h^1(\IC(3k-1))=1$. Note that $2k$ and $3k-1$ are not consecutive when $k\ge3$, so it easily follows that $C$ must also be arithmetically Buchsbaum.

In the case $k=2$, it is enough to show that the multiplication map $f\cdot :H^1(\IC(4))\to H^1(\IC(5))$
is zero for every $f\in H^0(\op3(1))$. To see this, consider the following commutative diagram
$$ \xymatrix{
H^1(\Omega^1_{\p3}) \ar[d]^{\simeq} \ar[r]^{\cdot f} & H^1(\Omega^1_{\p3}(1))=0 \ar[d] \\
H^1(\IC(4)) \ar[r]^{\cdot f} & H^1(\IC(5))
} $$
with the left vertical arrow being an isomorphism. It clearly follows that the lower horizontal map must vanish.
\end{proof}

The converse of Proposition \ref{nc=>dim2} requires a technical lemma about rank 2 bundles on $\p3$.

\begin{Lemma}\label{dim23}
There are no rank 2 locally free sheaves $E$ on $\p3$ such that $\dim_\C H^1_*(E)=2,3$.
\end{Lemma}
\begin{proof}
The claim is an immediate consequence of the various classification results in \cite{RV}. If $H^1_*(E)$ is concentrated in a single degree, then \cite[Example 1]{RV} implies that $\dim_\C H^1_*(E)=1$. If $H^1_*(E)$ is concentrated in two different degrees, then \cite[Proposition 4]{RV} implies that $\dim_\C H^1_*(E)=4$. Finally, $H^1_*(E)$ is concentrated in three different degrees, then \cite[Theorem 2]{RV} implies that $\dim_\C H^1_*(E)\ge4$.
\end{proof}

Our next result completes the proof of the first part of Main Theorem \ref{mthm4}.

\begin{Prop}
Let $\sF$ be a foliation by curves on $\p3$. If
$\dim_\C R_\sF=2$, then $N_\sF^*$ is a twisted null-correlation bundle.
\end{Prop}
\begin{proof}
By Lemma \ref{l:rao}, either $\dim_\C M_\sF=2$ or $\dim_\C M_\sF=1$. The first possibility is ruled out by Lemma \ref{dim23}, while second possibility implies, by \cite[Example 1]{RV}, that $N_\sF^*$ must be a twisted null-correlation bundle.
\end{proof}

\begin{Obs}
Lemma \ref{l:rao} and Lemma \ref{dim23} have another interesting consequence: there are no foliations by curves $\sF$ such that $\dim_\C R_\sF=3$. 

If $E$ is a stable rank 2 locally free sheaf with $c_1(E)=0$ and $c_2(E)=2$, then $E(-3)$ is the conormal sheaf of a foliation by curves $\sF$ of degree 3 (see item (3) of Main Theorem \ref{classification-degree3}) with $\dim_\C R_\sF=5$, since $\dim_\C M_\sF=4$ and $h^1(N_\sF^*)=h^1(E(-3))=0$.

However, we have not been able to find an example of a foliation by curves $\sF$ satisfying \linebreak $\dim_\C R_\sF=4$. \qed
\end{Obs}

Regarding the proof of Main Theorem \ref{mthm4}, we are left with the task of describing the moduli space of the foliations of degree $2k+1$ whose conormal sheaf is a twisted null-correlation bundle, as defined in display (\ref{nc-normal}). The strategy is to check the vanishing conditions required in the hypotheses of Lemma \ref{lemma-moduli}.

Being $N$ a null-correlation bundle, we already know that
$$ \ext^2(N_\sF^*,N_\sF^*)=\ext^2(N,N)=0. $$
Let us now consider
$$ \ext^1(N_\sF^*,\Omega^1_{\p3}) \simeq H^1(\Omega^1_{\p3} \otimes N(k+2)); $$
Tensoring the sequence in display \eqref{nc-defn} by $\Omega^1_{\p3}(k+2)$ and taking cohomology we have that
$$ H^1(\Omega^1_{\p3} \otimes N(k+2)) \simeq H^1(\Omega^1_{\p3}\otimes \Omega^1_{\p3}(k+3)), $$
since $h^p(\Omega^1_{\p3}(k+1)) = 0$ for $p=1,2$ and every $k\ge1$. Taking now the following resolution of the cotangent bundle
$$ 0\rightarrow \O(-3) \rightarrow \O(-2)^4 \rightarrow \O(-1)^6 \rightarrow \Omega^1_{\p3}(1) \rightarrow 0, $$
and tensoring it by $\Omega^1_{\p3}(k+2)$, it is straightforward to see that 
$$ h^1(\Omega^1_{\p3}\otimes \Omega^1_{\p3}(k+3))=0, $$
as required.

The next step is to compute the two terms in the formula for the dimension. Recall that
$$ \dim \ext^1(N_\sF^*,N_\sF^*) = \dim \ext^1(N,N) = 5, $$
and note that
$$ \dim \Hom(N_\sF^*,\Omega^1_{\p3}) = h^0(\Omega^1_{\p3}\otimes N(k+2)). $$
To compute the latter, consider the Euler short exact sequence tensored by $N(k+1)$, that is
$$ 0 \rightarrow \Omega^1_{\p3} \otimes N(k+2) \rightarrow N(k+1)^{\oplus 4} \rightarrow N(k+2) \rightarrow 0. $$
Recalling that 
$$ h^0(N(t)) = 2\binom{t+3}{3} - (t+2), $$
we have
$$ h^0(\Omega^1_{\p3}\otimes N(k+2)) = 
4h^0(N(k+1))-h^0(N(k+2)) = 
8 \binom{k+4}{3} - 2\binom{k+5}{3} - 3k - 8. $$
and hence the required dimension of the moduli space. Together with the considerations made in the paragraph below the proof of Lemma \ref{lemma-moduli}, the proof of Main Thorem \ref{mthm4} is finally complete.

\bigskip

At last, as a by-product of the results in this section, we close this paper by providing a full characterization of those foliations by curves $\sF$ such that $R_\sF$ is concentrated in a single degree.

\begin{Theorem}
Let $\sF$ be a foliation by curves such that $R_\sF$ is concentrated in degree $\delta$. Then
\begin{enumerate}
\item[(i)] either $\sF$ has degree $\delta+1$, $N_\sF^*$ splits as a sum of line bundles, and $\dim_\C R_\sF=1$;
\item[(ii)] or $\sF$ has degree $3$, $\delta=4$, $\dim_\C R_\sF=2$, and $N_\sF^* \simeq N(-3)$ for some null-correlation bundle $N$.
\end{enumerate}
\end{Theorem}

\begin{proof}
The hypothesis implies that $C:=\sing(\sF)$ is a curve; assume that $h^1(\IC(p))=0$ for every $p\ne\delta$, and $\dim_\C R_\sF=h^1(\IZ(\delta))=t\ne0$.

Consider the exact sequence
$$ 0 \to N_\sF^* \to \Omega^1 \to \IC(d-1) \to 0; $$
by hypothesis, $N_\sF^*$ is locally free. Since the map $H^1(\IC(p)) \to H^2(N_\sF^*(p-d+1))$ is surjective, we conclude that $H^2(N_\sF^*(k))=0$ for every $k\ne\delta-d+1$, and $h^2(N_\sF^*(\delta-d+1))\le t$.

If $h^2(N_\sF^*(\delta-d+1))=0$, then $N_\sF^*$ must split as a sum of line bundles by the Horrocks splitting criterion. It then follows that $h^1(N_\sF^*(k))=0$ for every $k$, so $H^1(\Omega^1(\delta-d+1)\simeq H^1(\IC(\delta))$, thus $d=\delta+1$, and $h^1(\IC(\delta))=1$.

Recall that every rank 2 locally free sheaf $E$ such that $h^2(E(k))=0$ for every $k\ne l$ and $h^2(E(l))\ne0$ must be of the form $N(-l-3)$ for some null-correlation bundle $N$ \cite[Example 1]{RV}, in which case $h^2(E(l))=1$.

Therefore, if $h^2(N_\sF^*(\delta-d+1))\ne0$, then $N_\sF^*=N(d-\delta-4)$ for some null-correlation bundle $N$; from the exact sequence
$$ H^1(N(p-\delta-3) \to H^1(\Omega^1_{\p3}(p-d+1) \to H^1(\IC(p)) \to H^2(N(p-\delta-3)) \to 0 $$
we conclude that $d=\delta+1$, so $N_\sF^*=N(-3)$, from which it follows that $d=3$ and $t=2$.
\end{proof}

\bibliographystyle{amsalpha}

\end{document}